\documentclass[sn-mathphys]{sn-jnl}% Math and Physical Sciences

\usepackage{enumerate}
\usepackage{amsmath,amssymb,amsfonts}
\usepackage{array}
\usepackage{multirow}
\usepackage{tabularx}
\usepackage{bm}
\usepackage{graphicx}

\theoremstyle{thmstyleone}%
\newtheorem{theorem}{Theorem}[section]
\newtheorem{lemma}[theorem]{Lemma}
\newtheorem{proposition}[theorem]{Proposition}
\newtheorem{corollary}[theorem]{Corollary}

\newtheorem{problem}[theorem]{Problem}
\newtheorem{example}[theorem]{Example}
\theoremstyle{thmstylethree}%
\newtheorem{definition}[theorem]{Definition}
\newtheorem{remark}[theorem]{Remark}

\numberwithin{equation}{section}
\allowdisplaybreaks
\AtBeginDocument{\let|\origbar}

\newcommand{\Aut}{\operatorname{Aut}}
\newcommand{\Sym}{\operatorname{Sym}}
\newcommand{\soc}{\operatorname{soc}}
\newcommand{\Gal}{\operatorname{Gal}}
\newcommand{\Sp}{\operatorname{Sp}}
\newcommand{\PSL}{\operatorname{PSL}}
\newcommand{\PGL}{\operatorname{PGL}}

\newcommand{\AGammaL}{\mathrm{A}\Gamma\mathrm{L}}

\newcommand{\rank}{\operatorname{rank}}
\newcommand{\spanop}{\operatorname{span}}

\newcommand{\Ccal}{\mathcal C}

\begin{document}

\title[]{The classification of endpoint-transitive graphs with geodesic condition}

\author[1]{\fnm{Ximin} \sur{Wang}}\email{wxm240325@163.com}
\equalcont{These authors contributed equally to this work.}

\author*[1]{\fnm{Zheng} \sur{Huang}}\email{huangzhengmath@163.com}
\equalcont{These authors contributed equally to this work.}

\abstract{We introduce endpoint \(k\)-path-transitive graphs, in which the automorphism group fixing two prescribed vertices pointwise acts transitively on the paths of length \(k\) joining them. We study the geodesic case, where \(k\) equals the distance between these vertices. We identify the union of these geodesics with the Hasse graph of a finite bounded graded poset and reduce it to proper blocks by complete cuts and matching compression. 
Assume
% that the internal layers are nonsingleton and 
that the induced group \(K\) acts faithfully and \(2\)-homogeneously on the first distance layer.
We first prove that every nontrivial normal subgroup of \(K\) is transitive on every internal layer if and only if \(\operatorname{soc}(K)\) is. 
Under this normal-basic condition, we classify the proper blocks in the affine case and the equal-width proper blocks with a \(2\)-transitive internal layer in the almost-simple case.
The nontrivial design interfaces are Paley or affine symplectic designs in the affine case, and projective designs, the \(2\)-\((11,5,2\)) design, the Higman–Sims design, or their complements in the almost-simple case. 
The proper blocks have reduced rank at most four, whereas almost-simple proper blocks can have arbitrarily large reduced rank without the equal-width condition. 
We also give a stabilizer factorization criterion for assembling blocks while preserving endpoint-geodesic transitivity.

}

\keywords{Endpoint-geodesic transitivity; graded posets; 2-homogeneous permutation groups.}

\pacs[MSC Classification]{05C25, 05E18, 20B15, 06A07, 05B05}

\maketitle

\section{Introduction}\label{sec:introduction}
%Shortest paths between two prescribed endpoints are often more rigid than the ambient graph: every vertex lying on a shortest path has a fixed sum of its distances to the two endpoints, and every edge lying on a shortest path joins consecutive distance layers.
%Sets of shortest paths and their single-vertex reconfiguration graphs have been studied systematically in the reconfiguration literature; see Asplund et al.~\cite{asplund2018}. If a group fixes the two endpoints and is transitive on all shortest paths between them, then this distance structure is equivalent to transitivity on flags of each fixed type. The graph-theoretic problem is thereby converted into a problem about a graded poset with a group action, whose local interfaces can be studied both by double counting and by permutation representations and design theory.
Communication between computing nodes is a central factor in the construction of large-scale systems for artificial intelligence and high-performance computing.
Recent systems for distributed language model training illustrate the importance of network topology and the distribution of traffic over alternative routes, for example\cite{Jiang2024}.
These routes need not all be shortest paths. 
Routing schemes may also use longer paths to access additional links and avoid congestion, as in FatPaths and the recent Spritz framework\cite{Besta2020} \cite{Bonato2026}.

This setting motivates a structural question about alternative routes between prescribed endpoints. 
Represent a bidirectional network by a finite simple graph \(\Gamma\), with each edge assigned unit cost. 
For distinct vertices \(A,B\) and a positive integer \(k\), let \(\mathcal P_k(A,B)\) denote the set of simple unoriented paths of length \(k\) with endpoints \(A,B\).
These paths have the same number of hops, but their positions in the network may differ. 
We study the case in which any two are related by an automorphism of \(\Gamma\) preserving \({A,B}\) and we provide the following definition.
\begin{definition}
	Suppose that \(\mathcal P_k(A,B)\neq\varnothing\), and let \(G\leq\operatorname{Aut}(\Gamma)_{A,B}\). 
	We say that \((\Gamma;A,B)\) is endpoint \((G,k)\)-path-transitive if \(G\) is transitive on \(\mathcal P_k(A,B)\). 
	When \(G=\operatorname{Aut}(\Gamma)_{A,B}\), we simply say that \((\Gamma;A,B)\) is endpoint \(k\)-path-transitive.
\end{definition}
This definition expresses symmetry among the paths available to a prescribed pair of endpoints and provides an idealized model of structural equivalence among paths of the same length.
It leads to the following problem.
\begin{problem}
	Classify the endpoint k-path-transitive graphs.
\end{problem}

%Equivalently, the pointwise stabilizer \(\operatorname{Aut}(\Gamma)_{{A,B}}\) is transitive on \(\mathcal P_k(A,B)\). We call this property endpoint \(k\)-path transitivity.

Restricting attention to paths with the minimum number of hops(in other words, distances) gives the condition \(k=d_\Gamma(A,B)\).
In this case, \(\mathcal P_k(A,B)\) consists precisely of the geodesics from \(A\) to \(B\).
Endpoint-geodesic transitivity therefore arises as the shortest-path case of the broader endpoint \(k\)-path problem. 
In the present paper, we study this case.

A geodesic in a graph is a shortest path between its endpoints. For \(G\leq \operatorname{Aut}(\Gamma)\) and \(1\leq s\leq \operatorname{diam}(\Gamma)\), the graph \(\Gamma\) is called \((G,s)\)-geodesic-transitive if \(G\) is transitive on the ordered \(i\)-geodesics for every \(1\leq i\leq s\). This condition lies between \(s\)-arc-transitivity and \(s\)-distance-transitivity. The relations between these three properties were studied by \cite{Jin2015}.

%Shortest paths between fixed endpoints have also been studied through reconfiguration graphs [1], whose vertices are shortest paths and whose adjacency records a change at a single internal vertex. In the present paper, we study the union of these paths and the action of an endpoint stabilizer on them.

Recent work has combined explicit classifications with reductions through normal quotients. 
\cite{Jun2024} classified the \(2\)-geodesic-transitive graphs of order \(p^n\), where \(p\) is prime and \(n\leq3\). 
\cite{Jin2023} obtained a normal quotient reduction for graphs of odd order and determined the possible quasiprimitive action types.
For longer geodesics, \cite{Jin2021} studied normal quotients of \(3\)-geodesic-transitive graphs, with particular attention to quotients of diameter at most two, while \cite{Jin2024} considered \(4\)-geodesic-transitive graphs of girth \(6\) or \(7\). 
Further reductions for \(5\leq s\leq8\), with girth \(2s-2\) or \(2s-1\), were obtained in \cite{Jun2025}.

There is a direct connection between geodesic transitivity and the action of an endpoint stabilizer. If \(\Gamma\) is \((G,s)\)-geodesic-transitive and \(d_\Gamma(A,B)\leq s\), then the pointwise stabilizer \(G_{A,B}\) is transitive on the geodesics from \(A\) to \(B\). We study this transitivity condition for a prescribed pair of endpoints. The question is how the shortest paths between these endpoints can intersect, and how their union is constrained by the induced group action.

Let \(\Gamma\) be a finite connected graph, let \(d_\Gamma(A,B)=k\), and let \(\Delta\) be the union of all \(A\)-\(B\) geodesics. We call \(\Delta\) the geodesic core. Let \(K\) be the permutation group induced on \(\Delta\) by a subgroup of \(\operatorname{Aut}(\Gamma)_{A,B}\), and suppose that \(K\) is transitive on the \(A\)-\(B\) geodesics. The distance layers
\[
V_i=\{x\in V(\Delta):d_{\Gamma}(A,x)=i\},
\qquad 0\leq i\leq k,
\]
are \(K\)-invariant, with \(V_0=\{A\}\) and \(V_k=\{B\}\).
The graph problem is thereby converted into a problem about a graded poset with a group action.
The resulting poset \(Q\) is bounded and graded, its Hasse graph is \(\Delta\), and its maximal chains are precisely the \(A\)-\(B\) geodesics. 
%Conversely, every finite bounded graded poset arises in this way from its Hasse graph.
We may therefore research with the action of \(K\) on the maximal chains of \(Q\).

An interface is the cover relation between consecutive internal layers.
Two elementary operations simplify the structure.
On one hand, a complete interface separates \(Q\) into an ordinal sum.
On the other hand, at a perfect-matching interface, each element of one layer has a unique corresponding element in the next,
Deleting the latter layer gives a bijection between the original and compressed maximal-chain sets. 
We split at all complete interfaces and compress all the matching interfaces. 
After adjoining a least and a greatest element to each segment, we obtain the proper matching-reduced blocks, or simply proper blocks. 
A proper block with \(r\) internal layers has reduced rank \(r+1\). These blocks are the objects of our classification.

Normal subgroups provide a further reduction. For \(N\unlhd K\), the \(N\)-orbits in each layer form a quotient poset \(Q_N\), and \(K/N\) remains transitive on its maximal chains. We call \((Q,K)\) \(K\)-normal-basic if \(Q_N\) is a chain for every nontrivial normal subgroup \(N\) of $K$. 

Maximal-chain transitivity makes \(K\) transitive on each layer, but the action may still preserve nontrivial partitions, complicating the analysis of how geodesics intersect.
The first layer \(V_1\), which records the possible first steps of an \(A\)-\(B\) geodesic, provides a natural place to impose further conditions. 
Requiring primitivity on \(V_1\) excludes nontrivial \(K\)-invariant partitions of these first steps.
Therefore, we raise the following problem:
\begin{problem}
	Classify the geodesic cores \(\Delta\) of endpoint \((G,k)\)-path-transitive graphs with \(k=d_\Gamma(A,B)\), for which the induced group \(K=G^{V(\Delta)}\) acts primitively on the first distance layer \(V_1\).
\end{problem}

Under this problem, the number of common upper elements in a fixed layer still vary with the pair.
So, we further consider the assumption that \(K\) acts faithfully and 2-homogeneously on \(V_1\).
Thus all unordered pairs of distinct first steps are equivalent.
Then we obtain the following theorems.

%We assume that \(K\) acts faithfully and 2-homogeneously on this layer
%Our first result expresses normal-basicness of the socle.

\begin{theorem}\label{theorem1.1}
Let $Q$ be a finite bounded graded poset obtained from an endpoint-geodesic core. Let $K\leq\Aut(Q)$ act transitively on the maximal chains.
Suppose that every internal layer has at least two elements and that the action of $K$ on the first layer $V_1$ is faithful and 2-homogeneous, where $|V_1|\geq3$. Then the following conditions are equivalent:
\begin{enumerate}
\item[(i)] for every $1\neq N\trianglelefteq K$, the normal quotient $Q_N$ is a chain;
\item[(ii)] every $1\neq N\trianglelefteq K$ is transitive on every internal layer;
\item[(iii)] $M=\soc(K)$ is the unique minimal normal subgroup of $K$ and is transitive on every internal layer.
\end{enumerate}
Moreover, either $M$ is a regular elementary abelian group or $M$ is a nonabelian simple group.
\end{theorem}
The two possibilities in Theorem \ref{theorem1.1} give the affine and almost-simple cases. 
For each of them, we obtain the following two classification theorems.

\begin{theorem}\label{thm:main-affine}
Under the hypotheses of Theorem~\ref{theorem1.1} and $K$-normal-basic.
\begin{enumerate}
\item[(a)] Suppose that $M$ is a regular elementary abelian group, if $K^{V_1}$ is 2-homogeneous but not 2-transitive, then every proper block has exactly one of the following two shapes:

\medskip
\noindent\textbf{S1 (one-layer type)}
one internal layer;

\noindent\textbf{S2 (Paley type)}
two internal layers whose unique interface is the Paley symmetric design $2$-$(q, \frac{q-1}{2},\frac{q-3}{4})$, where $q \equiv 3(mod~4) >3$.

%It is either of the one-layer type or of the Paley type over $\mathbb F_q$, where $q\equiv3\pmod4$ and $q>3$. Both types occur.
\item[(b)] Suppose that $M$ is a regular elementary abelian group, if $K^{V_1}$ is 2-transitive, then every proper block has exactly one of the following four shapes:

\medskip
\noindent\textbf{A1 (one layer)} one internal layer;

\noindent\textbf{A2 (co-matching)} two internal layers whose unique interface is $J-P$;

\noindent\textbf{A3 (one symplectic interface)} two internal layers whose interface is $S^\varepsilon(2m)$, with $m\geq2$;

\noindent\textbf{A4 (complementary symplectic pair)} three internal layers whose consecutive interfaces are $S^\varepsilon(2m)$ and its complementary dual $S^{-\varepsilon}(2m)$, with matrices $B$ and $J-B^{\mathsf T}$.

\item[(c)] Suppose that $M$ is nonabelian simple, if $K^{V_1}$ is 2-transitive, then there is no uniform bound on the reduced rank.
\end{enumerate}
\end{theorem}

In almost simple case, the internal layers of the proper blocks do not have equal width, a property that is automatic in the affine case.
So we therefore impose this equal-width condition to the almost simple case and continue our research.
In return, we relax 2-transitivity on one internal layer, with no restriction on its position.

\begin{theorem}\label{thm:main-almost-simple}
Suppose that $M$ is nonabelian simple.
%The natural action of $S_n$ on the Boolean lattice $B_n$ gives a family whose reduced rank tends to infinity.
Let $R$ is a proper block whose internal layers have equal width and \(K\leq\operatorname{Aut}(R)\) act transitively on its maximal chains.
Suppose that \(K\) acts 2-transitively on at least one internal layer.
Then the proper block has exactly one of four shapes:

\medskip
\noindent\textbf{AS1 (one layer)} one internal layer;

\noindent\textbf{AS2 (co-matching)} two internal layers whose interface is $J-P$;

\noindent\textbf{AS3 (one design)} two internal layers whose interface is a projective design, the $2$-$(11,5,2)$ design, or the Higman--Sims design, or the complement of one of these;

\noindent\textbf{AS4 (mixed complementary pair)} three internal layers whose consecutive interfaces have matrices $B$ and $J-B^{\mathsf T}$ for the same one of the preceding design cases.
\end{theorem}

The classifications in Theorems \ref{thm:main-affine} and \ref{thm:main-almost-simple} leave four possible arrangements of internal layers: a single layer, a co-matching, one design interface, or a design interface followed by its complementary dual. Figure \ref{fig:zonglan} illustrates these arrangements, using the Fano plane for the two design examples and the other designs in the classification give the same layer arrangements.

% TODO: \usepackage{graphicx} required
\begin{figure}[H]
	\centering
	\includegraphics[width=1\linewidth]{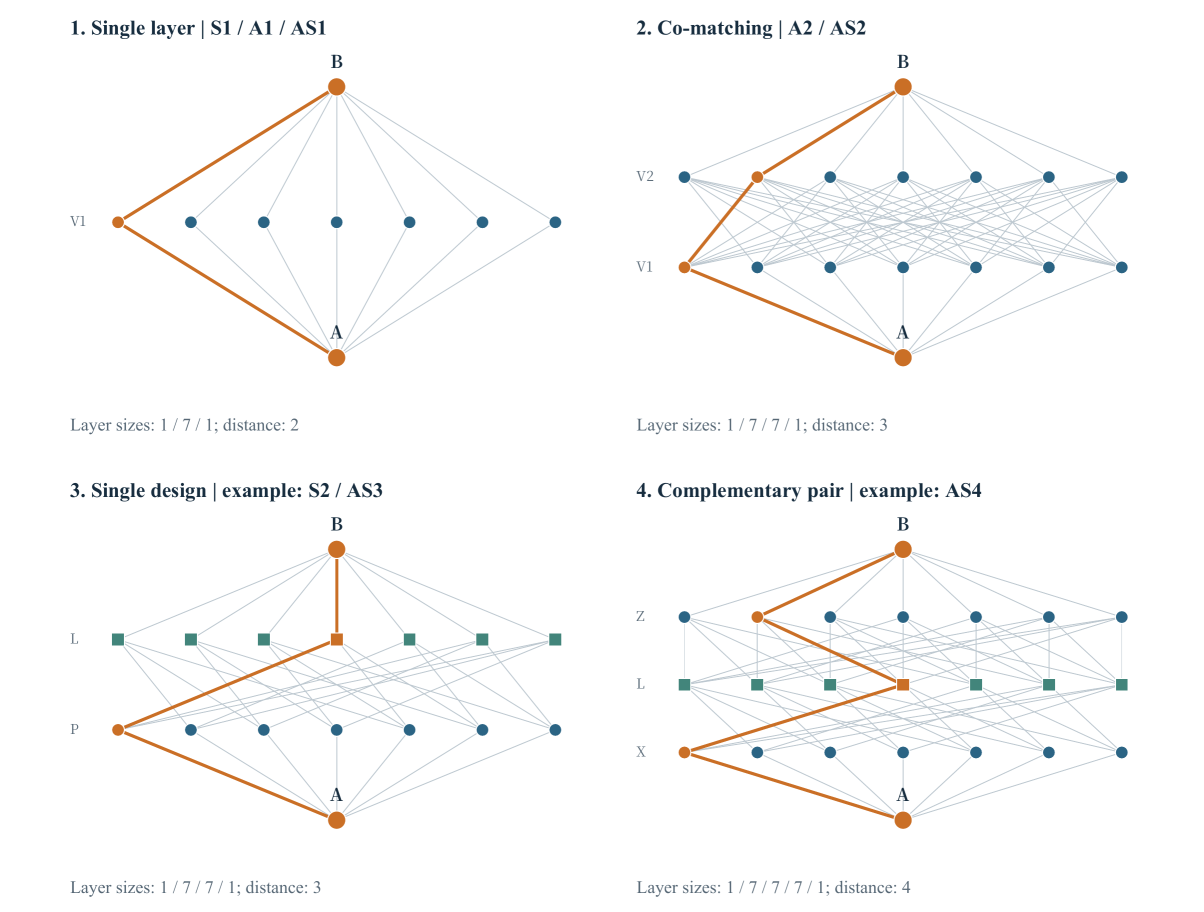}
	\caption{Examples of reduced blocks}
	\label{fig:zonglan}
\end{figure}

The preceding classifications give the following consequence for geodesic cores.
\begin{corollary}\label{cor:geodesic-block-reconstruction}
	Let $\Delta$ be the geodesic core of an endpoint $(G,k)$-path-transitive
	graph $(\Gamma;A,B)$, where $k=d_\Gamma(A,B)$.
	Suppose that 
	%the associated poset and induced group 
	it satisfies the hypotheses
	of Theorem~\ref{thm:main-affine}(a) or (b), or Theorem~\ref{thm:main-almost-simple}.
	Then $\Delta$ is obtained from the Hasse graphs of a finite ordered
	sequence of proper blocks of types above.
	respectively.
\end{corollary}

%\begin{proof}
%	By Proposition~\ref{thm:core-hasse} and
%	Corollary~\ref{cor:path-chain-transitivity}, $\Delta$ is the Hasse graph
%	of its associated graded poset, and the induced group is transitive on
%	the maximal chains. Propositions~\ref{prop:complete-cut-decomposition}
%	and~\ref{prop:matching-compression} reduce this poset to proper blocks
%	by complete-cut decomposition and matching compression. The induced
%	action on the maximal chains of each block remains transitive.
%	Under the stated hypotheses, the respective classification theorem
%	identifies these blocks. Reversing the two reductions recovers
%	$\Delta$, proving the assertion.
%\end{proof}

Section \ref{sec:preliminaries} records the group-theoretic and design-theoretic results used in the proofs. Section \ref{sec:3} establishes the correspondence with graded posets and develops the two reductions to proper blocks. Section \ref{sec:4} treats normal quotients and proves Theorem \ref{theorem1.1}. Sections \ref{sec:affine} and \ref{sec:almost-simple} prove the affine and almost-simple classifications, respectively.

%The proof does not rest on a single classification theorem, but on the combination of four mechanisms. First, maximal-chain transitivity makes every set of comparable tuples of a fixed rank type a single orbit; consequently, the row sums, column sums, and two-step path counts of the interface matrices are constant. Second, regularity of the affine socle identifies every layer with the same translation space and expresses the stabilizer index of a chain as a product of local orbit lengths. Third, the incidence matrix $B$ of an equal-width 2-transitive interface satisfies
%\[
%BB^{\mathsf T}=(\kappa-\lambda)I+\lambda J,
%\]
%and is therefore invertible; it gives an isomorphism between the permutation modules of adjacent layers. Finally, if $B,C$ are consecutive interface matrices and $BC=hD$, comparison of the scalar actions on the nonprincipal constituent and of the row sums gives a strict integral constraint. This constraint excludes two interfaces on the same side of a design and forces the surviving pair to be complementary.
%
%Section~\ref{sec:preliminaries} records the external classification results actually used and checks their hypotheses explicitly. Sections~\ref{sec:geodesic-core} and~\ref{sec:interfaces} develop the geodesic core, normal quotients, and reduction operations. Section~\ref{sec:affine} proves the affine classification, while Section~\ref{sec:almost-simple} treats the almost-simple branch.

\section{Preliminaries}\label{sec:preliminaries}
We collect in this section the notation and elementary facts as well as some technical lemmas. 
Some basic facts will be used in the sequel without further
reference.

\subsection{Permutation-group notation}\label{notation}

All groups considered in this paper are finite, and all graphs are finite,
simple and undirected. Let $G$ act on a finite set $\Omega$. We write
$G^\Omega$ for the permutation group induced by $G$ on $\Omega$, and
\[
G_{(\Omega)}:=\ker(G\curvearrowright\Omega)
\]
for the kernel of the action. Thus
\[
G^\Omega\cong G/G_{(\Omega)}.
\]
Unless the action under consideration is known to be faithful, we do not
identify $G$ with $G^\Omega$.

For $\alpha\in\Omega$, the stabilizer of $\alpha$ in $G$ is denoted by
$G_\alpha$. More generally, if $\Delta\subseteq\Omega$, then $G_\Delta$
and $G_{(\Delta)}$ denote the setwise and pointwise stabilizers of
$\Delta$, respectively. Subscripts separated by commas denote simultaneous
stabilizers; for instance,
\[
G_{\alpha,\Delta}=G_\alpha\cap G_\Delta.
\]
If $\mathcal{X}$ is a $G$-invariant family of subsets, chains or other
combinatorial objects, then $G_X$ denotes the setwise stabilizer of
$X\in\mathcal{X}$.

We write $N\lhd G$ to mean that $N$ is a normal subgroup of $G$, and
$\operatorname{soc}(G)$ for the socle of $G$, namely the product of all
minimal normal subgroups of $G$. A transitive permutation group $G$ on
$\Omega$ is called \emph{quasiprimitive} if every nontrivial normal
subgroup of $G$ is transitive on $\Omega$. We use the standard meaning of
\emph{primitive}. An action is \emph{semiregular} if every point
stabilizer is trivial, and \emph{regular} if it is both transitive and
semiregular.

We shall distinguish carefully between $2$-homogeneity and
$2$-transitivity. The action of $G$ on $\Omega$ is
\emph{$2$-homogeneous} if $G$ is transitive on the unordered $2$-subsets
of $\Omega$, and is \emph{$2$-transitive} if $G$ is transitive on the
ordered pairs of distinct elements of $\Omega$.

Finally, faithfulness always refers to the particular action being
discussed. Thus, for a $G$-invariant subset or layer $\Delta$,
the statement that $G$ acts faithfully on $\Delta$ means
\[
G_{(\Delta)}=1.
\]
In particular, faithfulness of an action on a larger $G$-set does not in
general imply faithfulness on one of its invariant subsets.

\subsection{Graded posets, flags and Hasse graphs}

We also recall the standard terminology for finite partially ordered sets (see [Chapter~3]\cite{Stanley2012}). Let $Q$ be a finite poset. For
$x,y\in Q$, we write $x<y$ if $x\leq y$ and $x\neq y$. We say that
$y$ \emph{covers} $x$, and write
\[
x\lessdot y,
\]
if $x<y$ and there is no $z\in Q$ such that $x<z<y$. Two elements of
$Q$ are said to be \emph{comparable} if one is at most the other. The
poset $Q$ is \emph{bounded} if it has a unique minimum element
$\hat{0}$ and a unique maximum element $\hat{1}$.

A finite bounded poset $Q$ is \emph{graded} if there exists a rank
function
\[
\rho:Q\longrightarrow \mathbb{Z}_{\geq 0}
\]
such that
\[
\rho(\hat{0})=0
\qquad\text{and}\qquad
x\lessdot y \Longrightarrow \rho(y)=\rho(x)+1.
\]
If $\rho(\hat{1})=n$, then $Q$ is said to have rank $n$, and we write
\[
\operatorname{rk}(Q)=n.
\]
For $0\leq i\leq n$, the $i$th \emph{rank layer} of $Q$ is
\[
V_i=V_i(Q):=\{x\in Q:\rho(x)=i\}.
\]
Thus
\[
V_0=\{\hat{0}\},
\qquad
V_n=\{\hat{1}\}.
\]
Every interval
\[
[x,y]:=\{z\in Q:x\leq z\leq y\},
\qquad x\leq y,
\]
is again graded, with rank $\rho(y)-\rho(x)$.

A \emph{chain} is a subset of $Q$ whose elements are pairwise
comparable. A chain is \emph{maximal} if it is maximal with respect to
inclusion. Since $Q$ is finite, every chain is contained in a maximal
chain. If $Q$ is bounded and graded of rank $n$, then every maximal
chain has the form
\[
\hat{0}=x_0\lessdot x_1\lessdot\cdots\lessdot x_n=\hat{1},
\qquad x_i\in V_i,
\]
and in particular has length $n$. We denote by
\[
\mathcal{C}(Q)
\]
the set of maximal chains of $Q$.

More generally, a chain
\[
x=x_0<x_1<\cdots <x_r=y
\]
is called \emph{saturated} if
\[
x_{j-1}\lessdot x_j
\qquad (1\leq j\leq r).
\]
Thus a maximal chain in a finite bounded graded poset is a saturated
chain from $\hat{0}$ to $\hat{1}$.

A \emph{flag} of $Q$ is a chain
\[
F=(x_{i_1}<x_{i_2}<\cdots <x_{i_t}),
\qquad
x_{i_s}\in V_{i_s},
\]
where
\[
0\leq i_1<i_2<\cdots <i_t\leq n.
\]
Its \emph{rank type} is
\[
\operatorname{type}(F)
:=\{i_1,i_2,\ldots,i_t\}.
\]
In particular, the maximal chains of $Q$ are precisely the flags of
rank type $\{0,1,\ldots,n\}$.

For a subset $S\subseteq Q$, we write $Q[S]$ for the subposet induced
by $S$, with the order inherited from $Q$. The \emph{Hasse graph}
$H(Q)$ is the underlying undirected graph of the Hasse diagram of
$Q$; explicitly,
\[
V(H(Q))=Q
\]
and
\[
E(H(Q))
=
\bigl\{\{x,y\}:x\lessdot y\text{ or }y\lessdot x\bigr\}.
\]
Since $Q$ is graded, every edge of $H(Q)$ joins two consecutive rank
layers.

We write $\operatorname{Aut}(Q)$ for the group of order automorphisms
of $Q$. Since $\hat{0}$ and $\hat{1}$ are unique and the cover
relation is preserved by every order automorphism, each element of
$\operatorname{Aut}(Q)$ preserves the rank function and hence every
rank layer $V_i$.

Finally, to distinguish the rank of a poset from the rank of a matrix,
we reserve $\operatorname{rk}(Q)$ for the former. Unless stated
otherwise, the rank of a matrix $M$ is taken over $\mathbb{C}$ and is
denoted by $\operatorname{rank}_{\mathbb{C}}(M)$.
The term \emph{rank type} always refers to the ranks of elements in
the underlying graded poset.

\subsection{Known external theorems}

\begin{lemma}\label{thm:faithful-two-homogeneous}\cite{kantor1969,kantor1972,Dixon1996}
	Let $G\leq\Sym(\Omega)$ be faithful and 2-homogeneous, with $|\Omega|\geq3$. Then $G$ is primitive and has a unique minimal normal subgroup $M$. Exactly one of the following holds:
	\begin{enumerate}
		\item[(i)] $M$ is elementary abelian and regular on $\Omega$; after identifying $\Omega$ with the additive group of $M$, one has $G=M\rtimes G_0$, an affine group;
		\item[(ii)]$M$ is nonabelian simple and $M\leq G\leq\Aut(M)$, so $G$ is almost simple.
	\end{enumerate}
	If $G$ is not 2-transitive, then necessarily $|\Omega|=q\equiv3\pmod4$, and $\Omega$ may be identified with $\mathbb F_q$ in such a way that
	\[
	G\leq\AGammaL_1(q),
	\]
	and the stabilizer of zero has exactly two orbits of equal length on $\mathbb F_q^\times$, namely the squares and the nonsquares .
\end{lemma}

To prevent a perfect matching or a co-matching from being included incorrectly among the nontrivial designs, we explicitly assume $2\leq\kappa\leq v-2$. The cases $\kappa=1$ and $\kappa=v-1$ will always be isolated before the theorem is applied.

\begin{lemma}\label{thm:two-transitive-symmetric-designs}\cite{kantor1985}
	Let $\mathcal D=(\Omega,\mathcal B)$ be a simple symmetric $2$-$(v,\kappa,\lambda)$ design with $2\leq\kappa\leq v-2$, and let $G\leq\Aut(\mathcal D)$ act faithfully and 2-transitively on the point set $\Omega$.
	After replacing $\mathcal{D}$ with the complement design if necessary, it is possible to make $v > 2\kappa$.
	Then $\mathcal{D}$ is isomorphic to one of the following:
	\begin{enumerate}
		\item[(i)] a projective point--hyperplane design, with
		\[
		v=\frac{q^d-1}{q-1},\qquad
		\kappa=\frac{q^{d-1}-1}{q-1},\qquad
		\lambda=\frac{q^{d-2}-1}{q-1},\qquad d\geq3;
		\]
		\item[(ii)] the unique $2$-$(11,5,2)$ design, with socle $\PSL_2(11)$;
		\item[(iii)] the Higman--Sims $2$-$(176,50,14)$ design, with socle $\mathrm{HS}$;
		\item[(iv)] 
		the affine symplectic design \(S^{-}(2m)\), for some \(m\geq 2\), with parameters $v=2^{2m},
		\kappa=2^{2m-1}-2^{m-1},
		\lambda=2^{2m-2}-2^{m-1}.$
%		\kappa=2^{2m-1}+\varepsilon\cdot2^{m-1},
%		\lambda=2^{2m-2}+\varepsilon\cdot2^{m-1}$,
		And $\operatorname{soc}(G)$ is elementary abelian
		and acts regularly on the point set.
	\end{enumerate}
	The first three cases are almost-simple cases and the fourth is an affine case.
\end{lemma}

\begin{remark}\label{rem:degenerate-design-interfaces}
	When $\kappa=1$, the design interface is a perfect matching; when $\kappa=v-1$, it is a co-matching. Both may admit a point-2-transitive automorphism group, but neither lies within the nondegenerate range of Lemma~\ref{thm:two-transitive-symmetric-designs}. The complete relation $\kappa=v$ is likewise split off before any classification theorem is applied.
\end{remark}

\section{Geodesic cores and reductions}\label{sec:3}
\subsection{Geodesic cores and graded posets}

Let $X$ be a graph, let $A,B\in V(X)$, and suppose that $d_X(A,B)=k$. Denote by $\mathcal G_X(A,B)$ the set of all $A$--$B$ endpoint geodesics, and let $\Delta=\Delta_X(A,B)$ be their union. Thus the vertices and edges of $\Delta$ are precisely those that occur in at least one $A$--$B$ geodesic.

\begin{lemma}\label{lem:core-distances}
	We have $d_\Delta(A,B)=k$. Moreover, if $x\in V(\Delta)$ occurs as $x=x_r$ on an $A$--$B$ geodesic $P=(x_0,x_1,\ldots,x_k), x_0=A, x_k=B$,
	then 
	\[d_\Delta(A,x)=r, d_\Delta(x,B)=k-r.\]
	
	In particular, $x$ occupies the $r$th position on every $A$--$B$ geodesic containing it. If $xy\in E(\Delta)$, then
	\[
	\bigl|d_\Delta(A,x)-d_\Delta(A,y)\bigr|=1.
	\]
\end{lemma}

\begin{proof}
	Since $\Delta$ is a subgraph of $X$, hence $d_\Delta(A,B)\geq d_X(A,B)=k$.
	On the other hand, $\mathcal G_X(A,B)$ is nonempty, and each of its members is included in $\Delta$. 
	Therefore $d_\Delta(A,B)\leq k$, and equality follows.
	
	 Let $P = (x_0,x_1,\ldots,x_k) \in \Delta$ is a geodesic and $x=x_r$, so $d_\Delta(A,x)\leq r$. If $d_\Delta(A,x)<r$, concatenate an $A$--$x$ path of length $d_\Delta(A,x)$ with the suffix $(x_r,x_{r+1},\ldots,x_k)$ of $P$. This gives an $A$--$B$ walk of length strictly less than $ k $.
	Deleting any closed subwalks from it yields an $A$--$B$ path still of length less than $k$, contrary to $d_X(A,B)=k$, thus $d_\Delta(A,x)=r$.
	Similarly, $d_\Delta(x,B)=k-r$.
	
	If $x$ occupies position $s$ on another $A$--$B$ geodesic, the result just proved gives both $s=d_\Delta(A,x) = r$. Finally, take $xy\in E(\Delta)$. By the definition of $\Delta$, some $A$--$B$ geodesic contains the edge $xy$. 
	The two endpoints of this side are in adjacent positions along this geodesic, so we obtain $|d_\Delta(A,x)-d_\Delta(A,y)|=1$.
\end{proof}

For $0\leq i\leq k$, define
\[
V_i:=\{x\in V(\Delta):d_\Delta(A,x)=i\}.
\]
%We call \(V_i\) the \(i\)-th distance layer from $A$.
By lemma~\ref{lem:core-distances}, $V_0=\{A\}$ and $V_k=\{B\}$, and $V_0,V_1,\ldots,V_k$ form a partition of $V(\Delta)$.

\begin{proposition}\label{thm:core-hasse}
	Define a relation $\preceq$ on $V(\Delta)$ as follows. If $x\in V_i$ and $y\in V_j$, let $x\preceq y$ if and only if either $x=y$, or some $A$--$B$ geodesic passes through $x$ before passing through $y$. Then
	\begin{equation}\label{eq:order-distance}
		x\preceq y\quad\Longleftrightarrow\quad i\leq j\ \text{and}\ d_\Delta(x,y)=j-i.
	\end{equation}
	The relation $\preceq$ makes
	\[
	Q(\Delta;A,B):=(V(\Delta),\preceq)
	\]
	a finite bounded graded poset of rank $k$, with rank function $\rho(x)=d_\Delta(A,x)$, whose Hasse graph is exactly $\Delta$. And the $A$--$B$ geodesics are naturally in bijection with the maximal chains of $Q(\Delta;A,B)$.
	
	Conversely, let $Q$ be a finite bounded graded poset of rank $k$, with unique minimum $A$ and unique maximum $B$. 
	Then the Hasse graph of $Q$, with endpoints $A,B$, is equal to its own geodesic core.
\end{proposition}

\begin{proof}
	We first prove \eqref{eq:order-distance}. Suppose that $x\preceq y$. Then some $A$--$B$ geodesic $P$ passes through $x$ before $y$. 
	By Lemma~\ref{lem:core-distances}, the positions of $x$ and $y$ on $P$ are $i$ and $j$, respectively, so $i\leq j$, and the $x$--$y$ subpath of $P$ has length $j-i$. 
	If $d_\Delta(x,y)<j-i$, replacing that subpath by a shorter $x$--$y$ path produces an $A$--$B$ walk of length less than $k$. Removing closed subwalks would then produce an $A$--$B$ path of length less than $k$, a contradiction. Hence $d_\Delta(x,y)=j-i$.
	
	Conversely, suppose that $i\leq j$ and $d_\Delta(x,y)=j-i$. 
	Choose shortest paths from $A$ to $x$, from $x$ to $y$, and from $y$ to $B$, of lengths $i, j-i, k-j$,respectively.
	Their concatenation is an $A$--$B$ walk of length exactly $k$. If this walk repeated a vertex, deleting the nonempty closed subwalk between two occurrences of that vertex would give an $A$--$B$ path of length less than $k$, contrary to $d_\Delta(A,B)=k$. Thus the walk is itself a path of length $k$, and hence an $A$--$B$ geodesic passing through $x$ and then $y$. 
	This proves \eqref{eq:order-distance}.
	
	The relation $\preceq$ is reflexive by definition.
	If $x\preceq y$ and $y\preceq x$, then the corresponding ranks satisfy both $i\leq j$ and $j\leq i$, so $i=j$. 
	Equation~\eqref{eq:order-distance} then gives $d_\Delta(x,y)=0$, so $x=y$. This proves antisymmetry. 
	Suppose that
	\[
	x\in V_i,\qquad y\in V_j,\qquad z\in V_\ell,
	\qquad x\preceq y\preceq z.
	\]
	Then
	\[
	d_\Delta(x,y)=j-i,\qquad d_\Delta(y,z)=\ell-j.
	\]
	The triangle inequality gives $d_\Delta(x,z)\leq\ell-i$. On the other hand,
	\[
	\ell=d_\Delta(A,z)\leq d_\Delta(A,x)+d_\Delta(x,z)=i+d_\Delta(x,z),
	\]
	so $d_\Delta(x,z)\geq\ell-i$. Hence equality holds, and \eqref{eq:order-distance} gives $x\preceq z$. Thus $\preceq$ is a partial order.
	
	Every $x\in V(\Delta)$ lies on an $A$--$B$ geodesic, and therefore $A\preceq x\preceq B$. 
	Since $V_0=\{A\}$ and $V_k=\{B\}$, the elements $A$ and $B$ are the unique minimum and maximum, respectively. If $xy\in E(\Delta)$ with $x\in V_i$ and $y\in V_{i+1}$, then \eqref{eq:order-distance} gives $x\prec y$.
	Since their ranks differ by one, no element can lie  between them, so $x\lessdot y$. 
	Conversely, suppose that $x\lessdot y$, with $x\in V_i$ and $y\in V_j$. Then $i<j$. If $j-i\geq2$, choose an endpoint geodesic passing through $x$ and $y$ in that order. 
	Then there is $z \in V_{i+1}$ satisfies $x\prec z\prec y$, contrary to $x\lessdot y$. Hence $j=i+1$, and \eqref{eq:order-distance} gives $d_\Delta(x,y)=1$, so $xy\in E(\Delta)$. Thus the Hasse graph of $Q(\Delta;A,B)$ is exactly $\Delta$, and every cover increases the rank by one.
	
	An $A$--$B$ geodesic passes through $V_0,V_1,\ldots,V_k$ in sequence, and consecutive vertices form cover relations; hence it gives a maximal chain. 
	Conversely, a maximal chain in a finite bounded poset contains $A$ and $B$, and consecutive elements are related by covers.
	Since each cover increases rank by one and the chain contains exactly one element of every rank, it gives an $A$--$B$ path of length $k$ in the Hasse graph.
	
	Now considering an finite bounded graded poset $Q$ of rank $k$, and let $\Gamma$ be its Hasse graph. Along every edge of $\Gamma$, the rank increases by exactly one.
	Therefore every $A$--$B$ path has length at least $\rho(B)-\rho(A)=k$. Any maximal chain supplies a path of length $k$, so $d_\Gamma(A,B)=k$.
	Finally, every two-element chain $x\lessdot y$ can be extended to a maximal chain in a finite poset, so every Hasse edge lies on an $A$--$B$ geodesic.
	Hence the union of all endpoint geodesics is exactly $\Gamma$, proving the converse.
\end{proof}

\begin{corollary}\label{cor:path-chain-transitivity}
	Let $K\leq\Aut(\Delta)$ fix $A$ and $B$ pointwise. Then $K$ preserves every distance layer $V_i$, and it is transitive on the set of $A$--$B$ geodesics if and only if it is transitive on the maximal chains of $Q(\Delta;A,B)$.
\end{corollary}

\begin{proof}
	Every $g\in K$ fixes $A$, so $d_\Delta(A,x^g)=d_\Delta(A,x)$, then $g$ preserves every $V_i$.
	Equation~\eqref{eq:order-distance} also shows that $g$ preserves $\preceq$ and hence induces a rank-preserving automorphism of $Q(\Delta;A,B)$.
	The bijection in Proposition~\ref{thm:core-hasse} between geodesics and maximal chains retains the vertex sequence. 
	On both paths and chains, an element $g\in K$ replaces each $x$ by $x^g$. 
	The bijection is therefore $K$-equivariant, and so the two transitivity conditions are equivalent.
\end{proof}

\begin{remark}\label{rem:two-kernels}
	Suppose that the endpoint stabilizer in the original graph is $\widehat K\leq\Aut(X)_{A,B}$. Then $\widehat K$ preserves the geodesic core $\Delta$. Throughout the paper, we first replace it by its faithful induced image on $\Delta$:
	\[
	K:=\widehat K/\widehat K_{(V(\Delta))}\leq\Aut(\Delta)_{A,B}.
	\]
	This replacement does not change the orbits on geodesics. 
\end{remark}

\subsection{Flag transitivity and incidence matrices}

\begin{lemma}\label{lem:fixed-type-flags}
	For every set of ranks $I=\{i_1<i_2<\cdots<i_t\}\subseteq\{0,1,\ldots,k\}$, the group $K$ is transitive on all flags of rank type $I$,
	\[
	x_{i_1}<x_{i_2}<\cdots<x_{i_t},\qquad x_{i_s}\in V_{i_s}.
	\]
	In particular, $K$ is transitive on every $V_i$; and for every $i<j$, the set
	\[
	\{(x,y)\in V_i\times V_j:x<y\}
	\]
	is a $K$-orbit.
\end{lemma}

\begin{proof}
	Let $F$ and $F'$ be two flags of rank type $I$. Since $Q$ is finite, each chain can be extended to an maximal chain. 
	Because $Q$ is graded and has a unique minimum and a unique maximum, such a chain contains exactly one element of every rank and is therefore a maximal chain.
	Extend $F$ and $F'$ to maximal chains $C$ and $C'$, respectively. By maximal-chain transitivity, some $g\in K$ satisfies $C^g=C'$.
	$K$ preserves ranks, while a maximal chain contains only one element of each rank. 
	Hence $g$ maps the rank-$i_s$ element of $C$ to the rank-$i_s$ element of $C'$, so $F^g=F'$.
	
	Taking $I=\{i\}$ proves transitivity on $V_i$, and taking $I=\{i,j\}$ proves that all comparable ordered pairs of those ranks form one orbit.
\end{proof}

For $0\leq i<k$, Lemma~\ref{lem:fixed-type-flags} allows us to define constants
\[
b_i:=|\{y\in V_{i+1}:x\lessdot y\}|\quad(x\in V_i),\qquad
c_{i+1}:=|\{x\in V_i:x\lessdot y\}|\quad(y\in V_{i+1}).
\]
Double counting the cover relations between $V_i$ and $V_{i+1}$ gives
\begin{equation}\label{eq:cover-balance}
	|V_i|b_i=|V_{i+1}|c_{i+1}.
\end{equation}
Let $A_i$ be the $0$--$1$ matrix whose rows are indexed by $V_i$ and columns are indexed by $V_{i+1}$, and whose entries are
\[
(A_i)_{x,y}=1\quad\Longleftrightarrow\quad x\lessdot y.
\]
For $0\leq i<j\leq k$, let $N_{ij}$ be the comparability matrix whose rows are indexed by $V_i$, whose columns are indexed by $V_j$, and whose entries are
\[
(N_{ij})_{x,y}=1\quad\Longleftrightarrow\quad x<y.
\]

\begin{proposition}\label{prop:chain-counting-matrices}
	For every $0\leq i<j\leq k$, there is a positive integer $h_{ij}$ such that
	\begin{equation}\label{eq:chain-matrix}
		A_iA_{i+1}\cdots A_{j-1}=h_{ij}N_{ij}.
	\end{equation}
	Moreover, the number of maximal chains is
	\begin{equation}\label{eq:max-chain-basic}
		|\Ccal(Q)|=\prod_{r=0}^{k-1}b_r=\prod_{r=1}^{k}c_r.
	\end{equation}
\end{proposition}

\begin{proof}
	Fix $x\in V_i$ and $y\in V_j$. 
	The $(x,y)$ entry of the matrix on the left of \eqref{eq:chain-matrix} is
	\[
	\sum_{x_{i+1}\in V_{i+1}}\cdots\sum_{x_{j-1}\in V_{j-1}}
	(A_i)_{x,x_{i+1}}(A_{i+1})_{x_{i+1},x_{i+2}}\cdots(A_{j-1})_{x_{j-1},y}.
	\]
	A product in this sum is $1$ exactly when
	\[
	x\lessdot x_{i+1}\lessdot x_{i+2}\lessdot\cdots\lessdot x_{j-1}\lessdot y.
	\]
	Thus the matrix entry is exactly the number of saturated chains from $x$ to $y$ in the interval $[x,y]$. 
	If $x$ and $y$ are incomparable, no such chain exists and the entry is $0$.
	If $x<y$, then the finite graded interval $[x,y]$ contains at least one saturated chain, so the entry is positive.
	
	Now let $(x',y')\in V_i\times V_j$ be another comparable pair. By Lemma~\ref{lem:fixed-type-flags}, there is a $g\in K$ satisfies $(x^g,y^g)=(x',y')$. The map
	\[
	(x=x_i\lessdot x_{i+1}\lessdot\cdots\lessdot x_j=y)
	\longmapsto
	(x^g=x_i^g\lessdot x_{i+1}^g\lessdot\cdots\lessdot x_j^g=y^g)
	\]
	is a bijection between the corresponding sets of saturated chains. 
	Hence the matrix entry depends only on the rank pair $(i,j)$, denote it by $h_{ij}$.
	This proves \eqref{eq:chain-matrix}.
	
	Finally, construct a maximal chain layer by layer, starting from the unique minimum. After reaching any point of $V_r$, there are exactly $b_r$ upper covers to choose.
	Every sequence of choices gives a maximal chain, and every maximal chain gives a unique such sequence. Therefore
	\[
	|\Ccal(Q)|=b_0b_1\cdots b_{k-1}.
	\]
	Starting instead from the unique maximum and choosing lower covers, then gives
	\[
	|\Ccal(Q)|=c_kc_{k-1}\cdots c_1.
	\]
	This proves \eqref{eq:max-chain-basic}.
\end{proof}

The following matrix identities will be used to recognize symmetric designs among the comparability relations.

\begin{lemma}\label{lem:interface-design-identity}
	Let $0<i<j<k$.
	\begin{enumerate}
		\item[(i)] Every row of $N_{ij}$ has the same sum $r_{ij}$, every column has the same sum $s_{ij}$, and
		\[
		|V_i|r_{ij}=|V_j|s_{ij}.
		\]
		\item[(ii)] If $|V_i|=|V_j|=v$ and $K$ is 2-transitive on $V_i$, then there are integers $d,\lambda$ such that
		\[
		N_{ij}N_{ij}^{\mathsf T}=(d-\lambda)I+\lambda J,
		\qquad d=r_{ij}=s_{ij}.
		\]
		Moreover, if $1\leq d\leq v-1$ and the columns of $N_{ij}$ are distinct, then they form a simple symmetric $2$-$(v,d,\lambda)$ design, and
		\[
		\lambda(v-1)=d(d-1).
		\]
		\item[(iii)] Under the hypotheses of~(ii) and $1\leq d\leq v-1$,  $N_{ij}$ is invertible over $\mathbb C$.
	\end{enumerate}
\end{lemma}

\begin{proof}
	Fix $x,x'\in V_i$ and there is a $g\in K$ with $x^g=x'$. Since $g$ preserves the order and stabilizes $V_j$, the map $y\mapsto y^g$ is a bijection from the elements of $V_j$ above $x$ to those above $x'$. Thus the row sum is constant.
	Similarly, the column sum is constant.
	Counting the entries equal to $1$ by rows gives $|V_i|r_{ij}$ and by columns gives $|V_j|s_{ij}$; both count the set
	\[
	\{(x,y)\in V_i\times V_j:x<y\},
	\]
	proving~(i).
	
	Suppose now that the layers have equal width and that $K$ is 2-transitive on $V_i$.
	Part~(i) gives $r_{ij}=s_{ij}=:d$. The $(x,x')$ entry of $N_{ij}N_{ij}^{\mathsf T}$ is the number of elements of $V_j$ lying above both $x$ and $x'$. 
	For $x=x'$, this number is $d$.
	For $x\neq x'$, $K$ is transitive on the ordered pairs $(x,x')$, therefore the number of common upper elements is  a constant $\lambda$. 
	This gives
	\[
	N_{ij}N_{ij}^{\mathsf T}=(d-\lambda)I+\lambda J.
	\]
	If the columns are distinct, take $V_i$ as the point set and the support of each column as a block. 
	There are $v$ distinct blocks of size $d$; every point lies in $d$ blocks, and every pair of distinct points lies in $\lambda$ blocks. 
	Thus this is a symmetric 2-design. Fixing a point $x$ and double counting $\{(x,C):x\neq x',\ x,x'\in C\}$ can prove the equation in (ii).
	%gives $d(d-1)$ by first choosing one of the $d$ blocks containing $x$ and then one of its other $d-1$ points, and gives $\lambda(v-1)$ by first choosing $x'$ and then one of the $\lambda$ blocks containing the pair.
	
	Finally, if $d=\lambda$, For any two row vectors $u,w$, the equalities
	\[
	\langle u,w\rangle=\langle u,u\rangle=\langle w,w\rangle=d
	\]
	force their supports to be equal, then all rows are identical. 
	And column-transitivity shows that every column is nonzero, and therefore the relation is complete, a contradiction. Thus $d-\lambda>0$. The Gram operator has eigenvalue $d-\lambda>0$ on $\mathbf1^\perp$ and eigenvalue $d^2>0$ on $\spanop\{\mathbf1\}$. It is positive definite, and $N_{ij}$ is invertible.
\end{proof}

\begin{lemma}\label{lem:comatching-isolation}
	Let three consecutive layers have equal width. Suppose that the first interface is co-matching and the second has degree at least two.
	Then the relation between outer layers is complete. 
	%In particular, whenever there is an independent reason that the comparability relation between the two outer layers is noncomplete.
\end{lemma}

\begin{proof}
	Suppose that the first interface is co-matching matrix $J-P$, the second interface has matrix $C$, and every row of $C$ has sum $b\geq2$.
	Since every column sum of $C$ is also $b$, we have $JC=bJ$. 
	The matrix $PC$ is obtained by permuting the rows of $C$, so its entries are $0$ or $1$. Thus every entry of $(J-P)C = bJ-PC$ is $b$ or $b-1$, and all of them are positive, which means that every pair $(x,z)$ is joined by at least one chain, so $x<z$. 
\end{proof}

\subsection{Reduction to proper blocks}

\begin{definition}\label{def:complete-cut}
	The $i$-th interface is called a \emph{complete cut} if $A_i=J$, that is, if every element of $V_i$ is covered by every element of $V_{i+1}$. For two bounded posets $P_-$ and $P_+$, their \emph{ordinal sum} $P_-\oplus P_+$ is their disjoint union, which retained the internal order and declared every element of $P_-$ is below every element of $P_+$.
\end{definition}

\begin{proposition}\label{prop:complete-cut-decomposition}
	Suppose that $A_i=J$, and put
	\[
	Q_-:=Q[V_0\cup\cdots\cup V_i],\qquad
	Q_+:=Q[V_{i+1}\cup\cdots\cup V_k].
	\]
	Then $Q=Q_-\oplus Q_+$. Restriction gives a bijection
	\[
	\Ccal(Q)\longrightarrow\Ccal(Q_-)\times\Ccal(Q_+),
	\]
	and this bijection is $K$-equivariant.
\end{proposition}

\begin{proof}
	Suppose that $x\in V_a$, $y\in V_b$ ($a\leq i<b$).
	Extend $x$ to be a maximal chain in $Q_-$, since $Q_-$ is graded, this yields some $u\in V_i$ with $x\leq u$. Similarly, we obtain some $v\in V_{i+1}$ with $v\leq y$. The complete cut gives $u<v$, then $x<y$. Thus $Q$ is exactly the ordinal sum.
	
	A maximal chain of $Q$ contains one element of each rank, and its restrictions to the two sides are maximal chains of $Q_-$ and $Q_+$. 
	Conversely, the union of any two local maximal chains is a maximal chain of $Q$, since every element on the left is comparable with every element on the right. 
	%The two operations are mutually inverse. 
	Since $K$ preserves every rank layer, it also preserves the restriction and union operations, so the bijection is $K$-equivariant.
\end{proof}

Proposition \ref{prop:complete-cut-decomposition} provides the decomposition of the maximal chain set, but its transitivity does not automatically hold. 
In fact, a complete cut converts the chain transitivity problem into an exact stabilizer-product problem.

\begin{proposition}\label{prop:double-coset-gluing}
	Suppose that $K$ is transitive on $\Ccal(Q_-)$ and $\Ccal(Q_+)$, respectively. Fix local maximal chains $C_-$ and $C_+$, and let
	\[
	L_-:=K_{C_-},\qquad L_+:=K_{C_+}
	\]
	be their stabilizers. 
	Then $K$ is transitive on $\Ccal(Q)$ if and only if
	\[
	K=L_-L_+.
	\]
\end{proposition}

\begin{proof}
	By Proposition~\ref{prop:complete-cut-decomposition}, it suffices to study the diagonal action of $K$ on $\Ccal(Q_-)\times\Ccal(Q_+)$. 
	$K$ is transitive on $\Ccal(Q_-)$, so every $K$-orbit contains a point of the form $(C_-,D_+)$. 
	Two such points $(C_-,D_+)$ and $(C_-,E_+)$ lie in the same $K$-orbit if and only if there is a $g\in L_-$ satisfies $D_+^g=E_+$. 
	Thus the $K$ is transitive on $\Ccal(Q)$ if and only if $L_-$ is transitive on $\Ccal(Q_+)$.
	
	Since $K$ is transitive on $\Ccal(Q_+)$, this set can be describe as the left coset space $K/L_+$. 
	The subgroup $L_-$ is transitive on $K/L_+$ if and only if every coset $gL_+$ lies in the $L_-$-orbit of $L_+$, that is, $g\in L_-L_+$. 
	Based on the arbitrariness of $g$, we have $K=L_-L_+$.
\end{proof}

The same criterion extends to any finite sequence of complete cuts.

Let $Q$ be a finite bounded graded poset of rank $k\geq2$, with endpoints
$A,B$ and rank layers $V_0,\ldots,V_k$, and $K\leqslant Aut(Q)$.
Choose
$
0=a_0<a_1<\cdots<a_t=k-1
$
such that the interface between $V_{a_j}$ and $V_{a_j+1}$ is complete
for $1\leq j<t$, and let
\[
P_j=Q[V_{a_{j-1}+1}\cup\cdots\cup V_{a_j}]
\qquad(1\leq j\leq t).
\]
Suppose that $K$ is transitive on each $\Ccal(P_j)$. Fix
$\Phi_j\in\Ccal(P_j)$, and define
\[
L_j=K_{\Phi_j},\qquad H_0=K,\qquad
H_j=\bigcap_{i=1}^{j}L_i\quad(1\leq j\leq t).
\]
\begin{corollary}\label{cor:multiple-complete-cuts}
	
	Restriction gives a $K$-equivariant bijection
	\[
	\Ccal(Q)\longrightarrow\prod_{j=1}^{t}\Ccal(P_j).
	\]
	Moreover, the following conditions are equivalent:
	\begin{enumerate}
		\item[(i)] $K$ is transitive on $\Ccal(Q)$;
		\item[(ii)] $K=H_{j-1}L_j$ for every $2\leq j\leq t$;
		\item[(iii)] $[K:H_t]=\displaystyle\prod_{j=1}^{t}[K:L_j] $.
	\end{enumerate}
%	For $t=1$, condition~\textup{(ii)} is empty. In all cases,
%	\[
%	|\Ccal(Q)|=\prod_{j=1}^{t}|\Ccal(P_j)|
%	=\prod_{j=1}^{t}[K:L_j].
%	\]
\end{corollary}

\begin{proof}
	Each maximal chain of $P_j$ contains one vertex of every layer in its part.
	Thus, considering a maximal chain of $Q$, its restriction to $P_j$ contains the maximal chain of $P_j$ and hence equals it.
%	Thus each maximal chain of $P_j$ contains one vertex of every layer in its part. 
	Conversely, the complete interfaces allow any choice of one such chain in each $P_j$ to concatenate, together with $A$ and $B$, to a maximal chain of $Q$. 
	Since $K$ preserves every rank layer, these two maps are $K$-equivariant and mutually inverse. 
	This proves the bijection.
	%and, by the orbit--stabilizer theorem on each factor, the chain-count formula.
	
	Let
	$\boldsymbol\Phi_j=(\Phi_1,\ldots,\Phi_j)$ and its stabilizer is
	$
	K_{\boldsymbol\Phi_j}
	=\bigcap_{i=1}^{j}K_{\Phi_i}
	=\bigcap_{i=1}^{j}L_i
	=H_j.
	$
	Let  $X_j=\prod_{i=1}^{j}\Ccal(P_i) = X_j=X_{j-1}\times\mathcal C(P_j)$.
	Whenever $K$ is transitive on $X_{j-1}$, then the
	Proposition~\ref{prop:double-coset-gluing} shows that $K$ is transitive
	on $X_j$ if and only if $H_{j-1}$ is transitive on $\Ccal(P_j)$.
	The orbit of $\Phi_j$ under $H_{j-1}$ has size
	\[
	[H_{j-1}:H_{j-1}\cap L_j].
	\]
	Since $K$ is transitive on $\Ccal(P_j)$, $|\Phi_j^K|=|\Ccal(P_j)|=[K:L_j]$, thus $H_{j-1}$ is transitive on $\Ccal(P_j)$ if and
	only if
	\[
	\frac{|H_{j-1}|}{|H_{j-1}\cap L_j|}
	=\frac{|K|}{|L_j|},
	\]
	that is to say $|H_{j-1}L_j|=|K|$. As $H_{j-1}L_j\subseteq K$, this
	is precisely $K=H_{j-1}L_j$.
%	If $K$ is transitive on $X_t$, its projections onto all $X_j$ are
%	transitive, so the preceding criterion proves~\textup{(ii)}.
%	Conversely, $K$ is transitive on $X_1$ by hypothesis, and
%	\textup{(ii)} gives transitivity on $X_2,\ldots,X_t$ successively.
	This proves the equivalence of~\textup{(i)} and~\textup{(ii)}.
	Finally, the orbit of $\boldsymbol\Phi_t$ has size $[K:H_t]$, while
	$|X_t|=\prod_{j=1}^{t}|\Ccal(P_j)|=\prod_j[K:L_j]$. That orbit is all of $X_t$ exactly when these
	two numbers are equal, proving the equivalence with~\textup{(iii)}.
\end{proof}

\begin{definition}\label{def:equivariant-matching}
	If $A_i$ is a permutation matrix, the $i$-th interface is called an \emph{perfect matching}. 
	There is then a unique bijection $\varphi:V_i\to V_{i+1}$ such that
	\[
	x\lessdot\varphi(x)\qquad(x\in V_i).
	\]
	%Since the cover relation is $K$-invariant, $\varphi$ is a $K$-equivariant bijection.
\end{definition}

Suppose that the $i$-th interface is an equivariant perfect matching. Delete $V_{i+1}$, replace every original pair of relations $x\lessdot\varphi(x)\lessdot z~(x\in V_i,\ z\in V_{i+2})$
by $x\lessdot z$, we obtain a $\overline Q$.
The inverse operation inserts an copy above a layer and joins corresponding points by a matching.

\begin{proposition}\label{prop:matching-compression}
$\overline Q$ is a graded poset of rank one less than $Q$, and the map
\[
(x_0,\ldots,x_i,\varphi(x_i),x_{i+2},\ldots,x_k)
\longmapsto
(x_0,\ldots,x_i,x_{i+2},\ldots,x_k)
\]
is a $K$-equivariant bijection $\Ccal(Q)\to\Ccal(\overline Q)$. 

	%it is called a \emph{matching subdivision}.
	
	%If $Q$ satisfies the hypotheses of Theorem~\ref{theorem1.1}, then matching compression and subdivision preserve maximum-chain transitivity, the action on the first layer provided the compression does not delete it, and the property that every nontrivial normal subgroup is transitive on all internal layers.
\end{proposition}

\begin{proof} 
	Treat $\varphi(x) \in V_{i+1}$ and its unique corresponding x $\in V_i$ as a single point.
	If it created a relation $u<u$, the reverse process would produce a closed directed walk in the original poset, a contradiction.
	Thus $\overline{Q}$ remains antisymmetric. 
	
	Every original maximal chain passes through a unique pair $x<\varphi(x)$.
	Deleting $\varphi(x)$ gives a chain of length one less. 
	Conversely,	inserting the unique point $\varphi(x)$ after $x$ in a maximal chain of $\overline{Q}$ yields a maximal chain of $Q$.  
	Thus the operations are mutually inverse. Every maximal chain of $\overline Q$ has the same length, and the new rank function is obtained by subtracting one above layer $i+1$, hence $\overline Q$ is graded.
	
	Equivariance follows directly from $\varphi(x^g)=\varphi(x)^g$. 
	%Maximum-chain transitivity is therefore equivalent on the two sides. For $N\trianglelefteq K$, transitivity of $N$ on $V_i$ is equivalent, via $\varphi$, to transitivity on $V_{i+1}$; all other layers are unchanged. Hence the all-layer normal-subgroup condition is preserved. Subdivision is the inverse construction and has the same properties.
\end{proof}

Thus matching compression lowers the rank while preserving the permutation action of \(K\) on maximal chains. Figure \ref{fig:3} illustrates this operation for two internal layers of size \(3\).
\begin{figure}[H]
	\centering
	\includegraphics[width=1\linewidth]{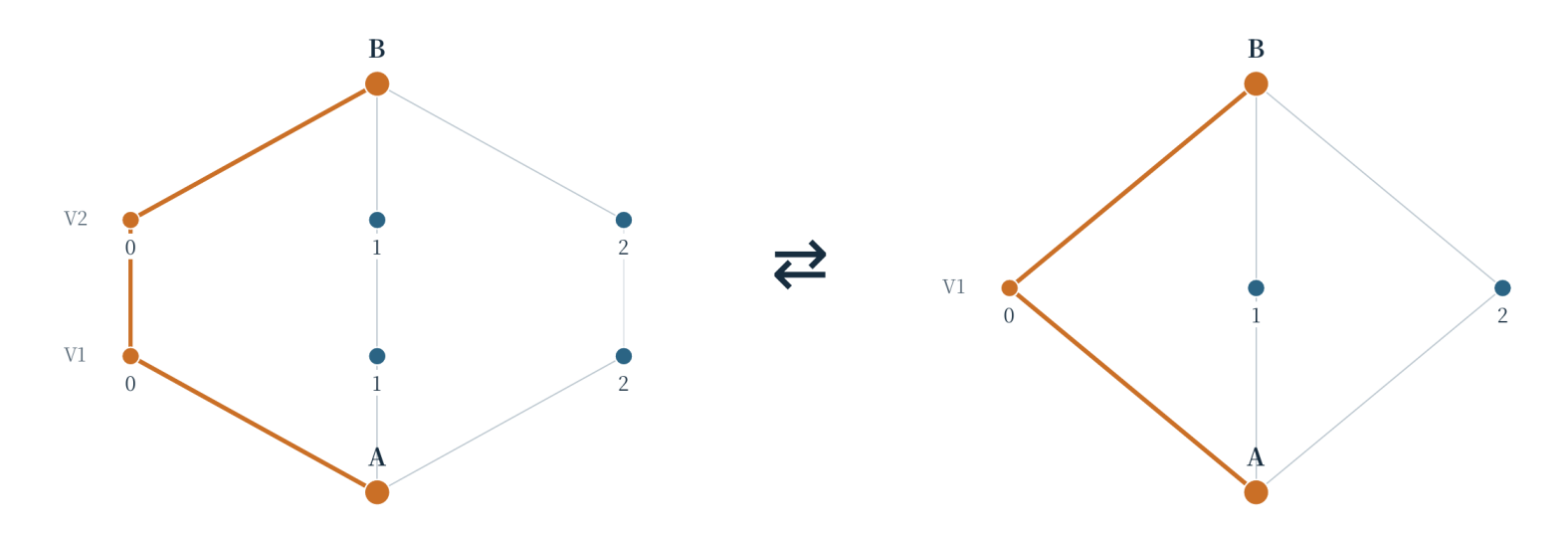}
	\caption{Illustration of compression and its inverse}
	\label{fig:3}
\end{figure}

So now we have two reversible reductions.
And we split $Q$ at every complete cut, obtaining consecutive segments of internal layers.
Then repeatedly compress all perfect matchings inside the segment until none remains.
This leads to the following definition.

\begin{definition}\label{def:proper-block}
	The resulting bounded segment is called a \emph{proper matching-reduced block}, for short \emph{proper block}. If it has $r$ internal layers, its \emph{reduced rank} is $r+1$.
\end{definition}
	
Thus a classification of local shapes, together with the stabilizer-product condition in Proposition~\ref{prop:double-coset-gluing}, records exactly the additional data needed to construct a global maximal-chain-transitive object from local blocks.
Accordingly, the classification theorems of this paper are stated for proper block.

\subsection{The proof of the Theorem \ref{theorem1.1}}\label{sec:4}

For $K$ preserves the ranks of $Q$ and is transitive on maximal chains. 
Let $N\trianglelefteq K$, we can obtain each $N$-orbit is contained in a $V_i$. 
Let $V_i/N$ denote the set of $N$-orbits on $V_i$, and define a relation on
\[
Q_N:=\bigcup_{i=0}^{k}V_i/N
\]
by
\[
X\leq_NY\quad\Longleftrightarrow\quad
\text{there exist }x\in X\text{ and }y\in Y\text{ with }x\leq y.
\]

%\subsection{The proof of Theorem~\ref{theorem1.1}}
\begin{proposition}\label{thm:normal-quotient-lifting}
	Under the preceding hypotheses, we have:
	\begin{enumerate}
		\item[(i)] $Q_N$ is a finite bounded graded poset of rank $k$, and the $i$-th layer is $V_i/N$;
		\item[(ii)] every maximal chain of $Q_N$ can lift to be a maximal chain of $Q$;
		\item[(iii)] $K/N$ preserves the rank of $Q_N$ and is transitive on the maximal chains of $Q_N$.
	\end{enumerate}
\end{proposition}

\begin{proof}
	We first show that $\leq_N$ is well defined. 
	If $x\leq y$ and $n\in N$, then $n$ is an automorphism of $Q$, so $x^n\leq y^n$, while $x^n$ and $y^n$ remain in the same $N$-orbits as $x$ and $y$. 
	Thus the chosen of x and y does not affect the relation.
	
	Since for any $x\in X$, $x\leq x$ gives $X\leq_NX$, the relation $\leq_N$ is reflexive.
	Suppose that $X\leq_NY$, $Y\leq_NZ$ and $x\in X, y,y'\in Y, z\in Z$ with $x\leq y$ and $y'\leq z$. 
	Based on $y$ and $y'$ lie in the same $N$-orbit, there is a $n\in N$ satisfies $y^n=y'$. Then
	\[
	x^n\leq y^n=y'\leq z,
	\]
	and $x^n\in X$, so $X\leq_NZ$. This proves transitivity. 
	If both $X\leq_NY$ and $Y\leq_NX$, then the ranks of these orbits satisfy both $\rho(X)\leq\rho(Y)$ and $\rho(Y)\leq\rho(X)$, hence are equal. 
	Take elements respectively such that  $x\leq y$ of equal rank satisfy $x=y$, and therefore $X=Y$. Hence $\leq_N$ is antisymmetric.
	
	Define $\rho_N(X)=i$ for $X\in V_i/N$. 
	This is well defined because $N$ preserves each rank layer. 
	The endpoints $A$ and $B$ are fixed by $K$, so each forms a $N$-orbit and they are the unique minimum and maximum of $Q_N$. 
	If $X<_NY$ and $\rho_N(Y)-\rho_N(X)\geq2$, choose representatives $x<y$. 
	Take a saturated chain in the finite graded interval $[x,y]$, and let $z\in Q$ be its element of rank $\rho(x)+1$. 
	Let $Z=z^N$, then
	\[
	X<_NZ<_NY.
	\]
	%The three orbits have distinct ranks and are therefore distinct. Thus 
	The comparable orbits of consecutive ranks have no possible intermediate rank and hence do form a cover. 
	Therefore $\rho_N$ is the rank function of $Q_N$. 
	Moreover, every maximal chain in finite bounded poset $Q_N$ passes through the rank $0,1,\ldots,k$.
	This proves part~(i).
%	Every inclusion-maximal chain in a finite bounded poset contains both endpoints and has consecutive elements related by covers; it consequently passes through the ranks $0,1,\ldots,k$ and has length $k$. This proves part~(i).
	
	Let
	\[
	X_0<_NX_1<_N\cdots<_NX_k
	\]
	be a maximal chain of $Q_N$, where $X_i\in V_i/N$, $x_{i-1}\in X_{i-1}$ and $x_0=A$.
	Since $X_{i-1}<_NX_i$, there exist $u_{i-1}\in X_{i-1}$ and $u_i\in X_i$ with $u_{i-1}<u_i$. 
	Choose $n_i\in N$ such that $u_{i-1}^{n_i}=x_{i-1}$, and 
	\(
	x_i:=u_i^{n_i}.
	\)
	Since $n_i$ preserves the order, $x_{i-1}<x_i$.
	Continuing to $i=k$ we have:
	\[
	A=x_0<x_1<\cdots<x_k=B.
	\]
	This is a maximal chain in $Q$, which is lifted from the quotient chain in $Q_N$, proving part~(ii).
	
	For $g\in K$, define
	\[
	(x^N)^g:=(x^g)^N.
	\]
	If $x^N=y^N$, then $y=x^n$ for some $n\in N$, and
	\[
	y^g=x^{ng}=x^{g(g^{-1}ng)}.
	\]
	Since $N\trianglelefteq K$, one has $g^{-1}ng\in N$, so $(x^g)^N=(y^g)^N$. Thus the action is well defined. 
	%It preserves $\leq_N$ and the ranks. 
	%Every $n\in N$ acts trivially on the set of $N$-orbits, so the action factors through $K/N$. Its kernel may be strictly larger than $N$, and no claim of faithfulness is made.
	
	Finally, take two maximal chains of $Q_N$ and lift them by part~(ii) to maximal chains $C,C'$ of $Q$. Maximum-chain transitivity supplies $g\in K$ with $C^g=C'$. Since $g$ preserves rank, we can get $K/N$ is transitive on the maximal chains of $Q_N$, proving part~(iii).
	%it sends the lifted element of each rank to the lifted element of the same rank on the other chain and therefore sends the first quotient chain to the second. Thus $K/N$ is transitive on the maximal chains of $Q_N$, proving part~(iii).
\end{proof}

\begin{definition}\label{def:normal-basic}
	%Assume that $|V_i|>1$ for every $1\leq i\leq k-1$. The pair 
	$(Q,K)$ is called \emph{$K$-normal-basic} if any $1\neq N\trianglelefteq K$ is transitive on every layer $V_i$. 
	That is to say every layer of $Q_N$ has only one element, so $Q_N$ is a chain of rank $k$.
\end{definition}
%
%\begin{theorem}\label{thm:socle-criterion}
%	Suppose that $K$ preserves the ranks of $Q$ and is transitive on its maximal chains, and that $|V_i|>1$ for all $1\leq i\leq k-1$. Assume that the action of $K$ on $V_1$ is faithful and 2-homogeneous, and put
%	\[
%	M:=\soc(K).
%	\]
%	Then $K$ has a unique minimal normal subgroup $M$, and the following conditions are equivalent:
%	\begin{enumerate}
%		\item[(i)] $(Q,K)$ is $K$-normal-basic;
%		\item[(ii)] every $1\neq N\trianglelefteq K$ is transitive on every internal layer $V_i$;
%		\item[(iii)] $M$ is transitive on every internal layer $V_i$.
%	\end{enumerate}
%	If these equivalent conditions hold, then the action of $K$ on every internal layer is faithful and quasiprimitive, and exactly one of the following occurs:
%	\begin{enumerate}
%		\item[(a)] $M$ is a regular elementary abelian group;
%		\item[(b)] $M$ is a nonabelian finite simple group and $M\leq K\leq\Aut(M)$.
%	\end{enumerate}
%\end{theorem}
Based on the above analysis, we present the proof of Theorem \ref{theorem1.1} :
	By Lemma~\ref{thm:faithful-two-homogeneous}, the faithful (2)-homogeneous action of \(K\) on \(V_1\) yields a unique minimal normal subgroup \(M=\operatorname{soc}(K)\). 
	Every non-trivial normal subgroup of the finite group \(K\) contains a minimal normal subgroup of \(K\), and therefore contains \(M\).
	
	By Proposition ~\ref{thm:normal-quotient-lifting}, the \(i\)-th layer of \(Q_N\) is \(V_i/N\). Hence \(Q_N\) is a chain if and only if \(N\) is transitive on every internal layer, proving the equivalence of (i) and (ii). 
	Taking \(N=M\) gives (ii) \(\Rightarrow\) (iii). 
	Conversely, if (iii) holds, then every non-trivial normal subgroup of \(K\) contains the layer-transitive subgroup \(M\), and is therefore transitive on every internal layer. Thus (iii) implies (ii).
	
	Then, suppose that these equivalent conditions hold. Fixing an internal layer \(V_i\), let
	\[
	R_i=\ker(K\curvearrowright V_i)
	\]
	This kernel is a normal subgroup of \(K\). From its definition, we can know it fixes \(V_i\) pointwise, but we also can get its transitivity from (ii). Since \(|V_i|>1\), it forces \(R_i=1\). 
	Thus \(K\) acts faithfully on \(V_i\), and (ii) gives quasiprimitivity of this action directly.
	Moreover, Lemma \ref{thm:faithful-two-homogeneous} gives precisely two possibilities: either \(M\) is elementary abelian and regular on \(V_1\), or \(M\) is nonabelian simple and
	$
	M\le K\le \operatorname{Aut}(M).
	$
	\qedhere

\begin{remark}\label{rem:first-layer-faithfulness}
	If one assumes only that $K^{V_1}$ is 2-homogeneous, without requiring the action $K\curvearrowright V_1$ to be faithful,
	the transitivity of $\soc{K}$ on every internal layer does not imply $k$-normal-basicness.
	% then the statement ``the socle is transitive on all internal layers'' does not imply normal-basicness in Theorem~\ref{theorem1.1}.
	
	Let
	\(
	Q=\{A\}\oplus V_1\oplus V_2\oplus\{B\}, |V_1|=3,|V_2|=2,
	\)
	where $\oplus$ denotes the ordinal sum, so every element of $V_1$ is below every element of $V_2$. Let
	\[
	K=S_3\times C_2,
	\]
	where $S_3$ acts naturally on $V_1$ and fixes $V_2$ pointwise, and $C_2$ acts naturally on $V_2$ and fixes $V_1$ pointwise. This action is faithful on the whole vertex set of $Q$, and the maximal chains are exactly
	\[
	(A,x,y,B),\qquad x\in V_1,\quad y\in V_2.
	\]
	Thus $K$ is transitive on the maximal chains, and $K^{V_1}\cong S_3$ is 2-transitive on $V_1$. However,
	\[
	\soc(K)=(A_3\times1)(1\times C_2)=A_3\times C_2.
	\]
	On one hand, the subgroup $A_3\times1$ is transitive on $V_1$, and $1\times C_2$ is transitive on $V_2$, so the full socle is transitive on both internal layers;
	On the other hand, the non-trivial normal subgroup $A_3\times1$ fixes $V_2$ pointwise and is not transitive there. Hence the pair is not $K$-normal-basic.
	
	This example shows that different minimal normal factors may be responsible for transitivity on different rank layers. 
	The transitivity of their product $\soc(K)$ on every layer does not ensure that each non-trivial normal subgroup is transitive on every layer. 
	Faithfulness on the first layer is the condition that rules out this phenomenon.
\end{remark}

Theorem \ref{theorem1.1} reduces the study of \(K\)-normal-basic pairs to the affine and almost-simple cases. We consider these in Sections \ref{sec:affine} and \ref{sec:almost-simple}, respectively.

\section{The affine case}\label{sec:affine}

Throughout this section, we assume that the equivalent conditions of Theorem \ref{theorem1.1} hold and that \(M=\operatorname{soc}(K)\cong C_p^d\) is elementary abelian and regular on $V_1$. Let \(v=|M|\).
We first establish regularity of \(M\) on every internal layer and show that \(2\)-homogeneity and \(2\)-transitivity pass from the first layer to all internal layers.
%Throughout this section, we assume that the equivalent conditions of Theorem~\ref{theorem1.1} hold and that the unique minimal normal subgroup
%\[
%M\cong C_p^d
%\]
%is elementary abelian and $v=|M|=p^d$. 
%We first prove that, in the affine case, all layers are indeed controlled regularly by the same translation group. This fact is the basis of all subsequent equal-width matrix arguments and chain-index arguments.

\begin{lemma}\label{thm:affine-layer-regularity}
	
	The group \(M\) acts regularly on every internal layer \(V_i\), and hence \(|V_i|=v\). Moreover, \(K^{V_i}\) is \(2\)-homogeneous for every internal layer \(V_i\). If \(K^{V_1}\) is \(2\)-transitive, then so is \(K^{V_i}\) for every internal layer \(V_i\).
	
%For every internal layer $V_i$, the group $M$ acts regularly on $V_i$, and hence $|V_i|=v$. Let $H=K/M$. After a base point $x_i$ has been chosen in each layer, the map $K_{x_i}\to H$ is an isomorphism. The conjugacy orbits of $K_{x_i}$ on $M\setminus\{1\}$ coincide with the orbits of $H$ on $M\setminus\{1\}$ and are independent of $i$. Consequently, if $K^{V_1}$ is 2-homogeneous, respectively 2-transitive, then $K^{V_i}$ is 2-homogeneous, respectively 2-transitive, for every internal layer.
\end{lemma}

\begin{proof}
Fix $x\in V_i$. Since $M$ is transitive on $V_i$, every $y\in V_i$ can be described as $y=x^m$. Since $M$ is abelian,
\[
M_y=M_{x^m}=M_x^m=M_x.
\]
Thus $M_x$ fixes every point of $V_i$, so $M_x$ is contained in the kernel of the action of $M$ on $V_i$. 
Conversely, every element of that kernel fixes all points of $V_i$, and in particular fixes $x$, so it belongs to $M_x$. Hence $M_x$ is exactly the kernel of the action on $V_i$. 
The proof of Theorem~\ref{theorem1.1} has shown that $K$, and therefore $M$, acts faithfully on every internal layer. Thus $M_x=1$.
Hence $M$ is therefore regular on each layer. The orbit--stabilizer theorem gives $|V_i|=|M|=v$.

Continue to fix $x_i\in V_i$. Regularity gives $M\cap K_{x_i}=1$. Since $M$ is transitive on $V_i$, for each $g\in K$ there exists $m\in M$ such that $x_i^{gm^{-1}}=x_i$. 
Hence $g\in K_{x_i}M=MK_{x_i}$. It follows that $K=M\rtimes K_{x_i}$, and the natural projection $K\to K/M=H$ restricts to an isomorphism on $K_{x_i}$.

If $c\in K_{x_i}$ has image $h$ in $H$, identify $V_i$ with \(M\) by \(m\mapsto x_i^m\), under this identification, \(M\) acts by translations, while each \(c\in K_{x_i}\) acts by conjugation:
\[
(x_i^m)^c=x_i^{c^{-1}mc} = x_i^{m^h}.
\]
Thus the orbits of the point stabilizer on $V_i\setminus\{x_i\}$ correspond exactly to the orbits of the conjugation action of $H$ on $M\setminus\{1\}$. 
Even when distinct complements $K_{x_i}$ are not $M$-conjugate, their elements differ by translation components whose inner conjugation action on the abelian group $M$ is trivial. 
Therefore this orbit structure is independent of $i$. 
A transitive group is 2-transitive if and only if one point stabilizer is transitive on the remaining points.
And a group is 2-homogeneity when the non-trivial suborbits of one point stabilizer consist either of one self-paired orbit or of two mutually paired orbits. 
Since these orbits and their pairing are induced in every layer by the same action $H\curvearrowright M\setminus\{1\}$, 2-homogeneity or 2-transitivity on the first layer propagates to every internal layer.
\end{proof}

Let $i<j$ be two internal ranks, and for $y\in V_j$, define the lower shadow
\[
D_y^{(i)}:=\{x\in V_i:x<y\},\qquad d_{ij}:=|D_y^{(i)}|.
\]
\begin{proposition}\label{prop:affine-interlayer-design}
If $2\leq d_{ij}\leq v-1$, then
\[
\left(V_i,\{D_y^{(i)}:y\in V_j\}\right)
\]
is a simple symmetric $2$-$(v,d_{ij},\lambda_{ij})$ design, where
\[
\lambda_{ij}(v-1)=d_{ij}(d_{ij}-1).
\]
Moreover, $M$ is regular both on the point set and on the block-index set, thus the design is isomorphic to the development of a difference set in \(M\).
% so the design is the translation development of a difference set in $M$.
\end{proposition}

\begin{proof}
By Lemma~\ref{lem:fixed-type-flags}, the group $K$ is transitive on $V_j$, so all blocks have the constant size $d_{ij}$. lemma~\ref{thm:affine-layer-regularity} gives that $K$ is 2-homogeneous on $V_i$.
Hence the number of blocks containing any two distinct points is a constant $\lambda_{ij}$.
Thus we obtain a $2$-design, but it allows for the occurrence of repeated blocks.

So we further analyze, if $D_y^{(i)}=D_z^{(i)}$, then equality of lower shadows defines a $K$-invariant equivalence relation on $V_j$. 
By lemma~\ref{thm:affine-layer-regularity}, $K^{V_j}$ is 2-homogeneous and therefore primitive, so this equivalence relation is either discrete or universal. 
In the universal case, all lower shadows equal one nonempty proper subset $D\subset V_i$. 
For every $g\in K$, the equality $D=D_y^{(i)}$ yields $D^g=D_{y^g}^{(i)}=D$, so $D$ is a nonempty proper invariant subset of $K^{V_i}$, a contradiction. 
Hence the blocks are distinct. 
The numbers of points and blocks are both $v$, and the design is therefore simple and symmetric. 
The parameter identity is the double-counting formula proved in Lemma~\ref{lem:interface-design-identity}.

Finally, lemma~\ref{thm:affine-layer-regularity} shows that $M$ acts regularly on both $V_i$ and $V_j$. Fix $y_0\in V_j$. As $m$ ranges over $M$, the point $y_0^m$ ranges over $V_j$, and invariance of the order gives
\[
D_{y_0^m}^{(i)}=(D_{y_0}^{(i)})^m.
\]
Thus all blocks are precisely the $M$-translates of one base block.
\end{proof}

\begin{lemma}\label{thm:affine-chain-index}
Fix a maximal chain $\Phi$ and put $L=K_\Phi$. For $1\leq i\leq k-2$, let $\delta_i$ be the common degree of the interface $V_i$--$V_{i+1}$. If $k=2$, the empty product is understood to be $1$. Then
\[
\prod_{i=1}^{k-2}\delta_i=[K/M:LM/M].
\]
In particular,
\[
\prod_{i=1}^{k-2}\delta_i\mid |K/M|.
\]
\end{lemma}

\begin{proof}
On one hand, A maximal chain is constructed by first choosing one point in $V_1$, which can be done in $v$ ways. 
Then, for any chosen point  once a point in $V_i$, 
we pick one of its \(\delta_i\) upper covers in \(V_{i+1}\).
There are \(\delta_i\) possible scenarios in \(V_{i+1}\).
This layer-by-layer procedure counts every maximal chain exactly and gives
\[
|\Ccal(Q)|=v\prod_{i=1}^{k-2}\delta_i.
\]
On the other hand, $K$ is transitive on maximal chains, and hence $|\Ccal(Q)|=[K:L]$. If $m\in M\cap L$, then $m$ fixes the point of $\Phi$ in $V_1$. Since $M$ is regular on $V_1$, this implies $m=1$. Therefore
\[
|LM|=|L||M|=v|L|
\]
and
\[
[K:L]=\frac{|K|}{|L|}=v\frac{|K/M|}{|LM/M|}=v[K/M:LM/M].
\]
In conclusion, we can obtain the equation.
Furthermore, the divisibility assertion holds.

\end{proof}

\subsection{The semiaffine non-2-transitive case}\label{subsec:semiaffine}

In this subsection, $K^{V_1}$ is assumed to be 2-homogeneous but not 2-transitive. By Lemma~\ref{thm:faithful-two-homogeneous}, let
\[
M=(\mathbb F_q,+),\qquad K=M\rtimes H\leq\AGammaL_1(q),\qquad q=p^e\equiv3\pmod4.
\]
Let $Q_0$ and $N_0$ denote, respectively, the sets of squares and nonsquares in $\mathbb F_q$, and let $s=(q-1)/2$.
%To avoid confusion with the poset $Q$, the square class is denoted only by $Q_0$ in this subsection.

\begin{lemma}\label{lem:semiaffine-complements}
Every complement of $M$ in $K$ is conjugate to $H$ by an element of $M$.
\end{lemma}

\begin{proof}
If $q=3$, then $H=1$ and the assertion holds. Assume that $q>3$, and let
\[
S:=H\cap\mathbb F_q^\times
\]
be the subgroup of pure scalars. The scalar group is normal in $\Gamma L_1(q)$, so $S\trianglelefteq H$. Suppose that $S=1$. Then the projection of $H$ into $\Gal(\mathbb F_q/\mathbb F_p)$ is injective, and hence $|H|\leq e$.
On the other hand, $H$ has a transitive orbit on the square class of size $s$, so $s\leq |H|\leq e$. 
If $e=1$, then $q\geq7$ and $s\geq3>1=e$.
If $e\geq3$, then $q\equiv3\pmod4$ forces $p\equiv3\pmod4$ and $e$ odd, so $q\geq3^e$ and $(3^e-1)/2>e$.
Both cases are contradictory. Thus $S\neq1$.

Every complement $C$ projects isomorphically onto $H=K/M$, and can therefore be written uniquely as the graph of a $1$-cocycle:
\[
C=\{(f(a),a):a\in H\},\qquad f(ab)=f(a)+af(b).
\]
Conjugation by a translation $t\in M$ replaces $f$ by $f_t(a)=f(a)+(1-a)t$.
Choose $1\neq s_0\in S$, since $1-s_0$ is an invertible scalar on $M$, we can choose $t$ such that $f_t(s_0)=0$.
And we still denote the conjugated cocycle as $f$.

For $s\in S$, due to the $ss_0=s_0s$ and the cocycle identity, we have:
\[
f(s)+sf(s_0)=f(ss_0)=f(s_0s)=f(s_0)+s_0f(s).
\]
So $(1-s_0)f(s)=0$, and therefore $f(s)=0$. 
Now take $a\in H$ and $1\neq s\in S$.
Let $s'=asa^{-1}\in S\setminus\{1\}$. 
Since $f(a^{-1})=-a^{-1}f(a)$, we obtain,
\[
\begin{aligned}
0=f(s')&=f(asa^{-1})\\
&=f(a)+af(sa^{-1})\\
&=f(a)+asf(a^{-1})\\
&=(1-asa^{-1})f(a)=(1-s')f(a).
\end{aligned}
\]
And the scalar $1-s'$ is invertible, so $f(a)=0$.
This holds for every $a\in H$, and hence $f=0$. 
Thus the conjugated complement is $H$.
\end{proof}

\begin{lemma}\label{prop:paley-dichotomy}
The degree of every adjacent internal interface belongs to $\{1,s\}$. Degree $1$ gives a perfect matching, while degree $s$ gives the translation Paley symmetric design
\[
2\text{-}\left(q,\frac{q-1}{2},\frac{q-3}{4}\right).
\]
Moreover, when $q=3$ and $s=1$, the Paley interface is a matching.
\end{lemma}

\begin{proof}
Fix a point $y$ in the upper layer. 
By fixed-type flag transitivity, $K_y$ is transitive on the lower neighbours of $y$. 
The group $K_y$ is a complement to $M$. By Lemma~\ref{lem:semiaffine-complements}, after a suitable change of translation for the lower layer we may assume that $K_y=H$. The orbits of $H$ on $\mathbb F_q$ are exactly $\{0\},Q_0,N_0$, and $Q_0$, $N_0$ both of size $s$. 
The lower neighbour set is nonempty, so its size is either $1$ or $s$.

When the degree is $1$, 
biregularity and equality of the layer sizes imply that the interface matrix is a perfect matching.
When the degree is $s$, Proposition~\ref{prop:affine-interlayer-design} gives a simple symmetric $2$-$(q,s,\lambda)$ design.
For \(z\in V_{i+1}\), write
\[
D_z=\{x\in V_i: x<z\}
\]
for the corresponding block. In the coordinates chosen above, the \(K_y\)-orbits on \(V_i\) are \({0},Q_0,N_0\). Since \(K_y\) is transitive on \(D_y\) and \(|D_y|=s>1\), we have \(D_y=Q_0\) or \(N_0\).
For \(a\in\mathbb F_q\), let \(t_a\in M\) act on \(V_i\) as \(x\mapsto x+a\). Every point of \(V_{i+1}\) has the form \(y^{t_a}\).
since \(M\) acts regularly on that layer. Moreover, preservation of the order gives
\[
D_{y^{t_a}}=(D_y)^{t_a}=D_y+a.
\]
Consequently, the complete block set is
\[
\mathcal B=\{D_y+a: a\in\mathbb F_q\}.
\]

%In other words, the base block is $Q_0$ or $N_0$, and the regular translation group $M$ generates all other blocks. 
And the parameter identity gives
\[
\lambda=\frac{s(s-1)}{q-1}=\frac{(q-1)(q-3)}{4(q-1)}=\frac{q-3}{4}.
\]
This also shows directly that the parameter is integral when $q\equiv3\pmod4$. For $q=3$, $s=1$, it is also a matching obviously.
\end{proof}

\begin{lemma}\label{lem:paley-at-most-one}
If $q>3$, There is at most one internal interface with a degree of $s$.
\end{lemma}

\begin{proof}
Suppose that there are $r$ such interfaces, and remaining interfaces have degree $1$, so from Lemma~\ref{thm:affine-chain-index} we can obtain $s^r\mid |H|$.
Let $H_1$ be the stabilizer of $1\in Q_0$. 
If the semilinear element $x\mapsto ax^\sigma$ fixes $1$, then $a=1$. 
Thus $H_1$ embeds into the field automorphism group, and $t:=|H_1|$ divides $e$. The orbit--stabilizer theorem gives
\[
|H|=|1^H||H_1|=st.
\]
The proof of Lemma~\ref{lem:semiaffine-complements} established that $s>e\geq t$. If $r\geq2$, then $s^2\mid st$, so $s\mid t$, contradicting $0<t<s$.
\end{proof}
%
%\begin{theorem}\label{thm:semiaffine-classification}
%	suppose that $M$ is a regular elementary abelian group, if $K^{V_1}$ is 2-homogeneous but not 2-transitive, the proper blocks is exactly of the following:
%%In the affine branch that is 2-homogeneous but not 2-transitive, after all matching interfaces have been repeatedly compressed, the entire internal segment is exactly one of the following:
%
%\medskip
%\noindent\textbf{S1 (one-layer type)}
%
%$
%A<V<B,\qquad |V|=q;
%$
%
%\noindent\textbf{S2 (Paley type)}
%
%$A<\mathcal P<\mathcal B<B$, where $q>3$ and the interface $\mathcal P$--$\mathcal B$ is the  Paley symmetric design of Proposition~\ref{prop:paley-dichotomy}.
%
%Moreover, the reduced rank is at most $3$. And each type actually exists.
%%The types S1 and S2, together with all their equivariant matching subdivisions, actually occur and satisfy all the original hypotheses.
%\end{theorem}

Based on the above results, we will now present the proof of Theorem \ref{thm:main-affine}(a) and demonstrate the existence of each type.
\begin{proof}
Lemma~\ref{thm:affine-layer-regularity} shows that all internal layers have equal size. 
Proposition~\ref{prop:paley-dichotomy} says that each interface is either a matching or a Paley interface, and Lemma~\ref{lem:paley-at-most-one} says that at most one genuine Paley interface can occur. 
If there is no Paley interface, compressing all matchings and we obtain one internal layer, giving S1. 
If there is only one Paley interface, compressing all other matchings and we get two internal layers, giving S2. 
For $q=3$, the Paley degree is $1$ and is has been compressed as a matching, and the situation is the same as that of S1.
%These two alternatives exhaust, and mutually distinguish, the cases of zero or one Paley interface.

We now prove existence. For every $q\equiv3\pmod4$, let
\[
K=(\mathbb F_q,+)\rtimes Q_0
\]
act on $\mathbb F_q$ by translations and square scalar multiplications. Two ordered pairs of distinct points lie in the same $K$-orbit if and only if their differences belong to the same square class, so the action is not 2-transitive. For an unordered pair, interchanging the two points replaces the difference $d$ by $-d$. Since $-1$ is a nonsquare, $d$ and $-d$ lie in the two different square classes. Square scalar multiplication together with interchange of the two entries therefore sends any unordered pair of distinct points to any other one, and the action is 2-homogeneous. Lemma~\ref{thm:faithful-two-homogeneous} then ensures that the translation group $M$ is the unique minimal normal subgroup.

For S1, maximal chains correspond bijectively to the points of $V$. The group $K$ is transitive on them and $M$ is regular, so Theorem~\ref{theorem1.1} gives normal-basicness.

For S2 with $q>3$, take two copies of $\mathbb F_q$ and define
\[
x<y\quad\Longleftrightarrow\quad x-y\in Q_0.
\]
Translations preserve differences and square scalars act transitively on $Q_0$. Hence $K$ is transitive on incident flags, which are precisely the maximal chains. The group $M$ is regular on both layers, and the first-layer action is faithful, 2-homogeneous, and not 2-transitive. The same theorem therefore gives normal-basicness. In the endpoint suspensions, maximal chains have lengths $2$ and $3$, respectively, and every cover edge belongs to a maximal chain. Thus the Hasse graphs are their own endpoint-geodesic cores. Proposition~\ref{prop:matching-compression} shows that equivariant matching subdivisions preserve maximal-chain transitivity and normal-basicness.
\end{proof}
The smallest S2 example occurs at \(q=7\). In this case the nonzero squares are \(1,2,4\), and the blocks
\(
L_t=t+\{1,2,4\}, t\in\mathbb F_7,
\)
form the Fano plane. Figure \ref{fig:5} shows the resulting block and its incidence matrix. The group
\[
K=\{x\mapsto ax+b\in\{1,2,4\},\ b\in\mathbb F_7\}
\]
is the group used in the S2 construction above.
\begin{figure}[H]
	\centering
	\includegraphics[width=1\linewidth]{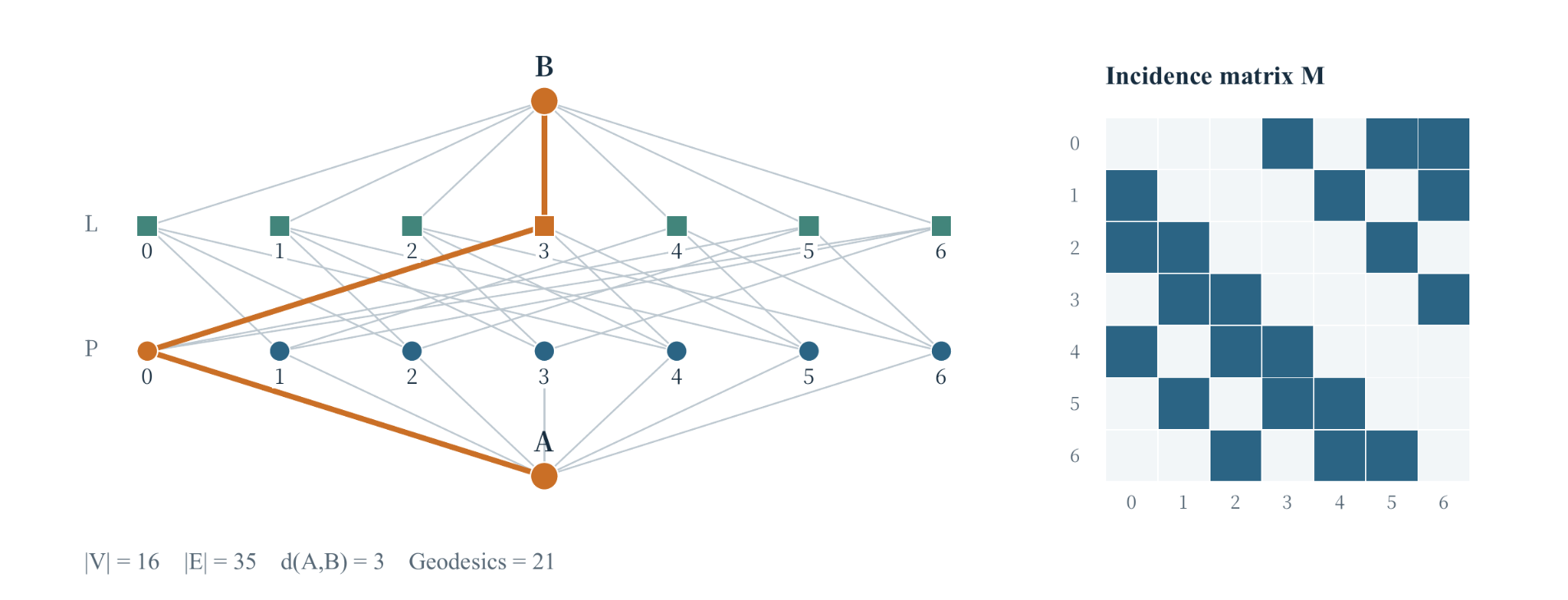}
	\caption{The smallest S2 example}
	\label{fig:5}
\end{figure}

\subsection{The affine 2-transitive case}\label{subsec:affine-two-transitive}

In this subsection, we assume that  $K^{V_1}$ is 2-transitive. 
By Lemma~\ref{thm:affine-layer-regularity}, $K$ is 2-transitive on every internal layer, and all internal layers have size $v$.

\begin{lemma}\label{lem:affine-invertibility-character}
In a consecutive segment of internal layers with no complete adjacent interface, every adjacent interface matrix is invertibleand the complex permutation modules of the layers are mutually $K$-isomorphic.
%, and their common permutation character is
%\[
%\pi=1_K+\chi,
%\]
%where $\chi$ is irreducible. 

Moreover, there are exactly two $K$-orbits on the Cartesian product of any two layers, and the comparable relation between them is that one of the non-empty non-complete orbit.
\end{lemma}

\begin{proof}
Let an adjacent interface have degree $a$. 
If $a=1$, its matrix is a permutation matrix;

If $a=v-1$, its complementary degree is $1$, so its matrix is $J-P$. On $\spanop\{\mathbf1\}$, the matrix $J-P$ has eigenvalue $v-1$. If $z$ lies in the orthogonal complement and $(J-P)z=0$, then $Pz=0$, and the invertibility of the permutation matrix $P$ gives $z=0$. Thus $J-P$ is also invertible.

If $2\leq a\leq v-2$, Proposition~\ref{prop:affine-interlayer-design} shows that its columns are distinct and form a simple symmetric design, and Lemma~\ref{lem:interface-design-identity} then gives proved the invertibility of the matrix from the design identity. 
Since each incidence relation is $K$-invariant, every interface matrix is a $K$-homomorphism between the two permutation modules. Invertibility makes it a $K$-module isomorphism. 
Composing these isomorphisms along the segment identifies all its layer modules.

Their common permutation character has the form $\pi = 1_K +\chi$, with $\chi$ irreducible.
For a 2-transitive permutation character $\pi$, one has $\langle\pi,\pi\rangle_K=2$, because this inner product implies the number of orbits of $K$ on ordered pairs: the diagonal orbit and the orbit of distinct pairs. 
The trivial representation has multiplicity one in a transitive permutation module. 
Hence the squared inner product $\pi-1_K$ is one ,that is to say it is an irreducible character $\chi$. 
The number of cross-orbits on two layers equals the inner product of their permutation characters and is $2$. Lemma~\ref{lem:fixed-type-flags} shows that the comparable pairs of two fixed ranks form one $K$-orbit. 
If this orbit were nonempty and complete, it would be the whole Cartesian product and would leave only one orbit, contrary to the condition of having $2$ orbits.
\end{proof}

We use the notation $S^\varepsilon(2m)$, and $\varepsilon\in\{+1,-1\}$, for the two complementary affine symplectic designs introduced in
Lemma \ref{thm:two-transitive-symmetric-designs}.
The affine case of Kantor's classification gives the following possibilities for a comparability relation.

\begin{proposition}\label{prop:affine-interface-candidates}
In a proper block, any non-empty non-complete comparability relation between two distinct internal layers is one of the following matrices:
\begin{enumerate}
\item[(i)] a permutation matrix;
\item[(ii)] $J-P$, where $P$ is a permutation matrix;
\item[(iii)] the incidence matrix of an affine symplectic design $S^\varepsilon(2m)$, where $m\geq2$ and $\varepsilon\in\{+1,-1\}$.
\end{enumerate}
For adjacent interfaces, case (i) is a matching which has been reduced.
\end{proposition}

\begin{proof}
By Lemma~\ref{lem:affine-invertibility-character}, the two layers have equal size, and the relation is non-complete. 
Let its common degree be $d$. 

If $d=1$, biregularity gives a permutation matrix. If $d=v-1$, the complementary matrix has degree $1$, and the matrix is $J-P$.

If $2\leq d\leq v-2$, Proposition~\ref{prop:affine-interlayer-design} gives a simple symmetric $2$-$(v,d,\mu)$ design with distinct columns. 
For $d=2$, the identity $\mu(v-1)=2$ forces $v=3$ and $\mu=1$, whence $d=v-1$, which belongs to case (ii) already. 
For $v-d=2$, its complementary design has block size 2. Apply the same parameter identity to the complementary design to obtain $v=3$, whence $d=1$, which belongs to case (i).
Thus we suppose that $3\leq\min\{d,v-d\}$.

The group $K$ is faithful and 2-transitive on points, and its socle is regular elementary abelian. The affine case of Lemma~\ref{thm:two-transitive-symmetric-designs} says that the design is $S^\varepsilon(2m)$ or its complement. 
In Kantor's quadratic-form construction, the complementary design is the symplectic design of the opposite Arf sign.
For $m=1$, the two parameter correspond respectively a matching and a co-matching, already covered by (i) and (ii). 
Thus $m\geq2$ in (iii). 
%The three cases exhaust all $1\leq d\leq v-1$ and are mutually exclusive.
\end{proof}

We assume that $u=2^{m-1}$, so that $v=2^{2m}=4u^2$. For $\varepsilon\in\{+1,-1\}$, let
\[
\kappa_\varepsilon=u(2u+\varepsilon),\qquad
\lambda_\varepsilon=u^2+\varepsilon u,\qquad
n:=\kappa_\varepsilon-\lambda_\varepsilon=u^2.
\]

Let $X,Y,Z$ be three consecutive internal layers of the same proper block, and suppose that the interfaces $X$--$Y$ and $Y$--$Z$ are $S^{\varepsilon_1}(2m)$ and $S^{\varepsilon_2}(2m)$, with incidence matrices $B$ and $C$. 
We obtain the following proposition:

\begin{proposition}\label{thm:symplectic-interface-rigidity}
The comparability relation $X$--$Z$ is a co-matching, $\varepsilon_2=-\varepsilon_1$, and
\[
BC=n(J-I),\qquad C=J-B^{\mathsf T}.
\]
\end{proposition}

\begin{proof}
Proposition~\ref{prop:chain-counting-matrices} gives a positive integer $h$ and the $X$--$Z$ comparability matrix $D$ such that $BC=hD$. Lemma~\ref{lem:affine-invertibility-character} ensures that the relation is non-complete, and Proposition~\ref{prop:affine-interface-candidates} proofs that $D$ is a matching, a co-matching, or a symplectic-design matrix.

The design identities for the two interfaces are
\[
BB^{\mathsf T}=nI+\lambda_{\varepsilon_1}J,\qquad
CC^{\mathsf T}=nI+\lambda_{\varepsilon_2}J.
\]
On one hand, since $BJ=\kappa_{\varepsilon_1}J$, we have
\[
\begin{aligned}
(BC)(BC)^{\mathsf T}
&=B(CC^{\mathsf T})B^{\mathsf T}\\
&=B(nI+\lambda_{\varepsilon_2}J)B^{\mathsf T}\\
&=n(nI+\lambda_{\varepsilon_1}J)+\lambda_{\varepsilon_2}\kappa_{\varepsilon_1}^2J\\
&=n^2I+\alpha J,
\end{aligned}
\]
On the other hand, $(BC)(BC)^{\mathsf T}=h^2DD^{\mathsf T}$. 
It is legitimate to compare the coefficients of $I$, because $I$ and $J$ are linearly independent for $v>1$.

If $D$ is a matching or a co-matching, then the coefficient of $I$ in $DD^{\mathsf T}$ is $1$, and hence $h=n$. Comparison of row sums gives
\[
hd=\kappa_{\varepsilon_1}\kappa_{\varepsilon_2},
\]
where $d$ is the degree of $D$. In the matching case $d=1$, but
\[
\kappa_{\varepsilon_1}\kappa_{\varepsilon_2}\geq u^2(2u-1)^2>u^2=n
\]
because $m\geq2$ and hence $u\geq2$, a contradiction. In the co-matching case $d=v-1=4u^2-1$, so we have
\[
\kappa_{\varepsilon_1}\kappa_{\varepsilon_2}=n(v-1)=u^2(4u^2-1).
\]
But
\[
\kappa_+\kappa_-=u^2(2u+1)(2u-1)=u^2(4u^2-1),
\]
whereas neither $\kappa_+^2$ nor $\kappa_-^2$ equals this number. Therefore $\varepsilon_2=-\varepsilon_1$.

If $D$ is a symplectic-design matrix, then the coefficient of $I$ in $DD^{\mathsf T}$ is $n$. Thus $h^2n=n^2$, and hence $h=u$. The row-sum identity require
\[
u\kappa_\eta=\kappa_{\varepsilon_1}\kappa_{\varepsilon_2}
\]
for some $\eta\in\{+1,-1\}$. The left side is at most $u^2(2u+1)$, while the righ side is at least $u^2(2u-1)^2$. For $u\geq2$,
\[
(2u-1)^2-(2u+1)=4u^2-6u=2u(2u-3)>0,
\]
a contradiction. Hence the relation can only be a co-matching.

The zero positions of that co-matching define the unique bijection between $X$ and $Z$. 
Labeling the layers through this bijection, we may take $D=J-I$. Since $h=n$, it follows that $BC=n(J-I)$. 
On the other hand,
\[
B(J-B^{\mathsf T})=\kappa_{\varepsilon_1}J-(nI+\lambda_{\varepsilon_1}J)
=(\kappa_{\varepsilon_1}-\lambda_{\varepsilon_1})J-nI=n(J-I).
\]
The matrix $B$ is invertible. Comparing the two identities and multiplying on the left by $B^{-1}$ gives $C=J-B^{\mathsf T}$.
\end{proof}

The co-matching between two non-adjacent layers will restrict the interface in a proper block.
As stated in the following proposition:

\begin{proposition}\label{lem:third-interface-exclusion}
There are no more interfaces before or after the two symplectic interfaces of Proposition~\ref{thm:symplectic-interface-rigidity}.
%No further interface can follow the two symplectic interfaces of Theorem~\ref{thm:symplectic-interface-rigidity}; nor can one precede them.
\end{proposition}

\begin{proof}
Suppose that there is an interface matrix $E$ on the right, with common degree $e\geq2$. By Proposition~\ref{thm:symplectic-interface-rigidity},
\[
BCE=n(J-I)E=n(eJ-E).
\]
Every entry of $E$ is $0$ or $1$, and $e\geq2$, so every entry on the right is at least $n(e-1)>0$. Therefore 
%every pair consisting of a point in the first layer and a point in the fourth layer is joined by a three-step increasing chain; that 
the distant comparability relation is complete, which contradicts with Lemma~\ref{lem:affine-invertibility-character}. 
The case of an interface on the left follows by transposing matrices.
\end{proof}

%\begin{theorem}\label{thm:affine-classification}
%Let $R$ be a proper matching-reduced block. If $v=2$, then $R$ has only type A1. If $v\geq3$, then $R$ has exactly one of the following four shapes:
%
%\medskip
%\noindent\textbf{A1 (one layer)} one internal layer;
%
%\noindent\textbf{A2 (co-matching)} two internal layers whose unique interface is $J-P$;
%
%\noindent\textbf{A3 (one symplectic interface)} two internal layers whose interface is $S^\varepsilon(2m)$, with $m\geq2$;
%
%\noindent\textbf{A4 (complementary symplectic pair)} three internal layers whose consecutive interfaces are $S^\varepsilon(2m)$ and its complementary dual $S^{-\varepsilon}(2m)$, with matrices $B$ and $J-B^{\mathsf T}$.
%
%Moreover the reduced rank is at most $4$. All four shapes exist. %Their endpoint suspensions and arbitrary equivariant matching subdivisions satisfy all the requirements of being a geodesic core, maximum-chain transitive, faithful and 2-transitive on the first layer, and normal-basic.
%\end{theorem}

Based on the above results, we will now present the proof of Theorem  \ref{thm:main-affine}(b) and demonstrate the existence of each type.
\begin{proof}
We first prove necessity. If $v=2$, every non-empty non-complete interface has degree $1$ and is a matching, so only one layer remains after compression. 

So, we assume $v\geq3$. Classify by the number of internal interfaces. With no interface, one obtains A1. 
With one interface, Proposition~\ref{prop:affine-interface-candidates} proves that it is either a co-matching or a symplectic interface with $m\geq2$, giving A2 and A3. 
If there are at least two interfaces, Lemma~\ref{lem:comatching-isolation} would make the two outer layers completely comparable, while Lemma~\ref{lem:affine-invertibility-character} excludes that distant complete relation, so a co-matching cannot be adjacent to an interface of degree at least $2$.
Hence the first two interfaces are both symplectic designs. Proposition~\ref{thm:symplectic-interface-rigidity} forces them to have opposite signs and to be complementary, and Lemma~\ref{lem:third-interface-exclusion} excludes a third interface. This gives A4. 
The cases of zero, one, and at least two interfaces are exhaustive and disjoint, so the classification is complete.

We next construct each type.

For A1 with $v=2$, take $K=C_2$ acting regularly on a two point set $X$. 
This action is faithful and 2-transitive, and $K$ itself is its unique minimal normal subgroup. 
For A1 with $v\geq3$, take any faithful affine 2-transitive group $K=M\rtimes H$ of degree $v$ acting on $X=M$. 
In either case, suspend the endpoints.
A maximal chain correspond bijectively to a point of $X$, so $K$ is transitive on them. 
The unique minimal normal subgroup is regular on the unique internal layer, and Theorem~\ref{theorem1.1} gives normal-basicness.

For A2, take two equivalent $K$-sets $X,Z$, definite $K$-equivariant $\varphi:X\to Z$, and define $x<z$ if and only if $z\neq\varphi(x)$. 
A maximal chain correspond bijectively to  an ordered pair of distinct points of $X$, so 2-transitivity of $K$ gives chain transitivity. 
The group $M$ is regular on both layers, so it is also normal-basic.

To construct A3 and A4, let $W$ be a $2m$-dimensional vector space over $\mathbb F_2$, where $m\geq2$, and let $\beta$ be a non-degenerate alternating bilinear form. Choose a non-degenerate quadratic form $q_\varepsilon$ with polar form $\beta$, normalizing the sign so that
\[
\sum_{x\in W}(-1)^{q_\varepsilon(x)}=\varepsilon2^m.
\]
%Thus $\varepsilon=(-1)^{\operatorname{Arf}(q_\varepsilon)}$.
Let $D_\varepsilon=\{x:q_\varepsilon(x)=0\}$. Then
\[
|D_\varepsilon|=2^{2m-1}+\varepsilon2^{m-1}=\kappa_\varepsilon.
\]
For $a\neq0$, through
$q_\varepsilon(x+a)=q_\varepsilon(x)+q_\varepsilon(a)+\beta(x,a)$ we have:
\[
\begin{aligned}
|D_\varepsilon\cap(D_\varepsilon+a)|
&=\frac14\sum_{x\in W}
\bigl(1+(-1)^{q_\varepsilon(x)}\bigr)
\bigl(1+(-1)^{q_\varepsilon(x+a)}\bigr)\\
&=\frac14\left(2^{2m}+2\varepsilon2^m+(-1)^{q_\varepsilon(a)}
\sum_{x\in W}(-1)^{\beta(x,a)}\right)\\
&=2^{2m-2}+\varepsilon2^{m-1}=\lambda_\varepsilon.
\end{aligned}
\]
Since $a\neq0$ and $\beta$ is non-degenerate, $x\mapsto\beta(x,a)$ is a nonzero linear functional, whose additive character sum is zero. Note that $\kappa_\varepsilon-\lambda_\varepsilon=u^2>0$, if $D_\varepsilon+a=D_\varepsilon+b$ with $a\neq b$, let $c=a+b\neq0$. Translation invariance of intersection sizes would give
\[
\kappa_\varepsilon=|D_\varepsilon+a|=|(D_\varepsilon+a)\cap(D_\varepsilon+b)|
=|D_\varepsilon\cap(D_\varepsilon+c)|=\lambda_\varepsilon,
\]
contrary to $\kappa_\varepsilon>\lambda_\varepsilon$. 
Hence these $2^{2m}$ blocks are distinct:
\[
\mathcal B_\varepsilon:=\{D_\varepsilon+a:a\in W\}.
\]
Each block has size $\kappa_\varepsilon$. For $x\neq y$, the number of blocks containing both $x$ and $y$ is
\[
\begin{aligned}
|\{a\in W:x,y\in D_\varepsilon+a\}|
&=|(D_\varepsilon+x)\cap(D_\varepsilon+y)|\\
&=|D_\varepsilon\cap(D_\varepsilon+x+y)|=\lambda_\varepsilon.
\end{aligned}
\]
Therefore $(W,\mathcal B_\varepsilon)$ is a simple $2$-$(2^{2m},\kappa_\varepsilon,\lambda_\varepsilon)$ design. 
It has equally points and blocks and is thus the symmetric design $S^\varepsilon(2m)$. 
The translation $t_w:x\mapsto x+w$ sends $D_\varepsilon+a$ to $D_\varepsilon+(a+w)$, so the translation subgroup $W$ is regular on both points and blocks.

We next verify explicitly that the linear symplectic group preserves this set of blocks. 
Take $g\in\Sp(W,\beta)$ and define
\[
q_\varepsilon^g(x):=q_\varepsilon(g^{-1}x).
\]
The polar forms of $q_\varepsilon^g$ and $q_\varepsilon$ are both $\beta$, so their sum is a linear functional on $W$.
Since $\beta$ is non-degenerate, there is a unique $a_g\in W$ such that
\[
q_\varepsilon^g(x)=q_\varepsilon(x)+\beta(a_g,x)\qquad(x\in W).
\]
The substitution $x=gy$ shows that $q_\varepsilon^g$ and $q_\varepsilon$ have the same character sum. 
On the other hand, the quadratic form identity gives
\[
\sum_{x\in W}(-1)^{q_\varepsilon(x)+\beta(a_g,x)}
=(-1)^{q_\varepsilon(a_g)}\sum_{x\in W}(-1)^{q_\varepsilon(x)}.
\]
$\sum_{x\in W}(-1)^{q_\varepsilon(x)}=\varepsilon2^m\neq0$, so $q_\varepsilon(a_g)=0$. Consequently,
\[
q_\varepsilon^g(x)=q_\varepsilon(x+a_g),\qquad g(D_\varepsilon)=D_\varepsilon+a_g.
\]
Thus $\Sp(W,\beta)$ preserves $\mathcal B_\varepsilon$, and the affine group
\[
G=W\rtimes\Sp(W,\beta)
\]
acts on points and blocks and preserves incidence.

We also verify the required properties on the group.
Given nonzero vectors $x,y\in W$, extend each of them to a symplectic basis of $W$.
The linear map sending the first basis to the second preserves $\beta$, belongs to $\Sp(W,\beta)$, and sends $x$ to $y$. 
Hence $\Sp(W,\beta)$ is transitive on $W\setminus\{0\}$ and $G$ is 2-transitive on $W$. 
If an affine transformation $t_ag$ acts identically on $W$, substituting $0$ gives $a=0$, and then $g=1$. Thus the action is faithful.

The natural $\mathbb F_2\Sp(W,\beta)$-module $W$ is irreducible: a nonzero invariant subspace ccontain all nonzero vectors of $W$ because of transitivity of the symplectic group on nonzero vectors. 
Hence $W$ is a minimal normal subgroup of $G$. 
Conversely, let $1\neq N\trianglelefteq G$. Then $N\cap W$ is an $\Sp(W,\beta)$-invariant subgroup of $W$. 
Since $W$ is elementary abelian, it is a subspace of the natural module and therefore $N\cap W$ equals $1$ or $W$. 
If $N\cap W=1$, then
\[
[N,W]\leq N\cap W=1,
\]
so $N$ centralizes the regular translation group on the point set $W$. 
The centralizer in a symmetric group of a regular group is its right regular representation. 
Since $W$ is abelian, its right and left regular representations coincide, and therefore $C_{\Sym(W)}(W)=W$. 
Hence $N\leq W$, contradicting $N\cap W=1\neq N$. 
Thus every non-trivial normal subgroup contains $W$, and
\[
\soc(G)=W.
\]

Finally, we prove the maximal chain transitivity. 
Each $g\in\Sp(W,\beta)$ maps $D_\varepsilon$ to a unique translated block, so $G_{D_\varepsilon}\to\Sp(W,\beta)$ is a bijective and hence an isomorphism. 
Moreover,
\[
G_{0,D_\varepsilon}=O(q_\varepsilon),
\]
because an affine transformation fixing $0$ is linear, and a symplectic transformation fixes $D_\varepsilon$ if and only if it fixes $q_\varepsilon$. 
The standard order formula\cite{kantor1975} gives
\[
[\Sp_{2m}(2):O_{2m}^\varepsilon(2)]=2^{m-1}(2^m+\varepsilon)=\kappa_\varepsilon.
\]
%Here $m\geq2$, $\beta$ is non-degenerate, and $q_\varepsilon$ has the prescribed Arf type, so all hypotheses are satisfied.
The orbit of $0$ under $G_{D_\varepsilon}$ lies inside $D_\varepsilon$ and has length exactly $\kappa_\varepsilon=|D_\varepsilon|$. 
Hence the block stabilizer is transitive on the points of the block. 
Together with transitivity of $W$ on the block set, this proves that $G$ is flag-transitive on the design. 
Suspending two endpoints gives A3. Maximal chains are precisely the design flags, and hence $G$ is transitive on them. 
In sum, we have already proved that the first-layer action is faithful and 2-transitive, that $\soc(G)=W$, and that $W$ is regular on both internal layers above, then Theorem~\ref{theorem1.1} therefore gives normal-basicness.

For A4, take $X=Z=W$ and $Y=\mathcal B_\varepsilon$, and define
\[
x<Y_0\quad\Longleftrightarrow\quad x\in Y_0,\qquad
Y_0<z\quad\Longleftrightarrow\quad z\notin Y_0.
\]
If $B$ is the first interface matrix, then the second is $C=J-B^{\mathsf T}$, and
\[
BC=B(J-B^{\mathsf T})=\kappa_\varepsilon J-(nI+\lambda_\varepsilon J)=n(J-I).
\]
Thus no block exists when $x=z$, and $n$ blocks exist when $x\neq z$ and the $X$--$Z$ relation is exactly a co-matching.

A maximal chain is uniquely determined by a triple $(x,Y_0,z)$ with $x\in Y_0$ and $z\notin Y_0$. 
Since $G$ is transitive on flags $(x,Y_0)$ of design, the first two entries of any such triple can be sent to $(0,D_\varepsilon)$. 
The third entry is then a nonsingular vector satisfying $q_\varepsilon(z)=1$. The stabilizer
\[
G_{0,D_\varepsilon}=O(q_\varepsilon)
\]
is transitive on all nonsingular vectors. Indeed, if $q_\varepsilon(z)=q_\varepsilon(z')=1$, then the map $z\mapsto z'$ preserves the quadratic-form values on the corresponding one-dimensional subspaces,
Witt's extension theorem for quadratic spaces in characteristic $2$ extends this isometry to an isometry of the entire nondegenerate quadratic space, yielding an element still belongs to $O(q_\varepsilon)$. \cite[Theorem 3.4]{sprehn2020}. 
%The present space is finite-dimensional and has nondegenerate polar form, so all hypotheses are satisfied.
Therefore $G$ is transitive on the maximal chains. Its translation socle $W$ is regular on all three layers, and its first-layer action is faithful and 2-transitive, so Theorem~\ref{theorem1.1} gives normal-basicness.

Finally, we prove the second interface have the opposite sign.
Since $\kappa_\varepsilon<|W|$, there is a $b\in W$ with $q_\varepsilon(b)=1$. Define
\[
q_{-\varepsilon}(x):=q_\varepsilon(x)+\beta(b,x).
\]
This quadratic form still has polar form $\beta$, and
\[
\sum_{x\in W}(-1)^{q_{-\varepsilon}(x)}
=(-1)^{q_\varepsilon(b)}\sum_{x\in W}(-1)^{q_\varepsilon(x)}=-\varepsilon2^m.
\]
Thus its Arf sign is $-\varepsilon$. Since
$q_\varepsilon(x+b)=q_\varepsilon(x)+q_\varepsilon(b)+\beta(x,b)$,
\[
\{x:q_{-\varepsilon}(x)=0\}=\{x:q_\varepsilon(x+b)=1\}=(W\setminus D_\varepsilon)+b.
\]
Therefore all translates of the zero set of $q_{-\varepsilon}$ are
\[
\{W\setminus(D_\varepsilon+a):a\in W\},
\]
which are exactly the complements of the blocks of the first interface. 
Hence the second interface is truly $S^{-\varepsilon}(2m)$.

%All four constructions are finite bounded graded posets, and every cover edge extends to a maximal chain. By the converse direction of Theorem~\ref{thm:core-hasse}, their Hasse graphs are precisely the geodesic cores of their two endpoints. Proposition~\ref{prop:matching-compression} completes the verification for arbitrary equivariant matching subdivisions.
\end{proof}
The complementary pair constructed above is illustrated in Figure \ref{fig:512} for \(m=2\) and \(\varepsilon=-1\). The first interface is the symmetric \(2\text{-}(16,6,2)\) design, and the second is its complementary dual, of degree \(10\). Writing \(N\) for the first incidence matrix, the product identity becomes
\[
N(J-N^{\mathsf T})=4(J-I).
\]
Thus each pair of outer elements with distinct labels has exactly four lines, while elements with the same label are incomparable.
\begin{figure}[H]
	\centering
	\includegraphics[width=1\linewidth]{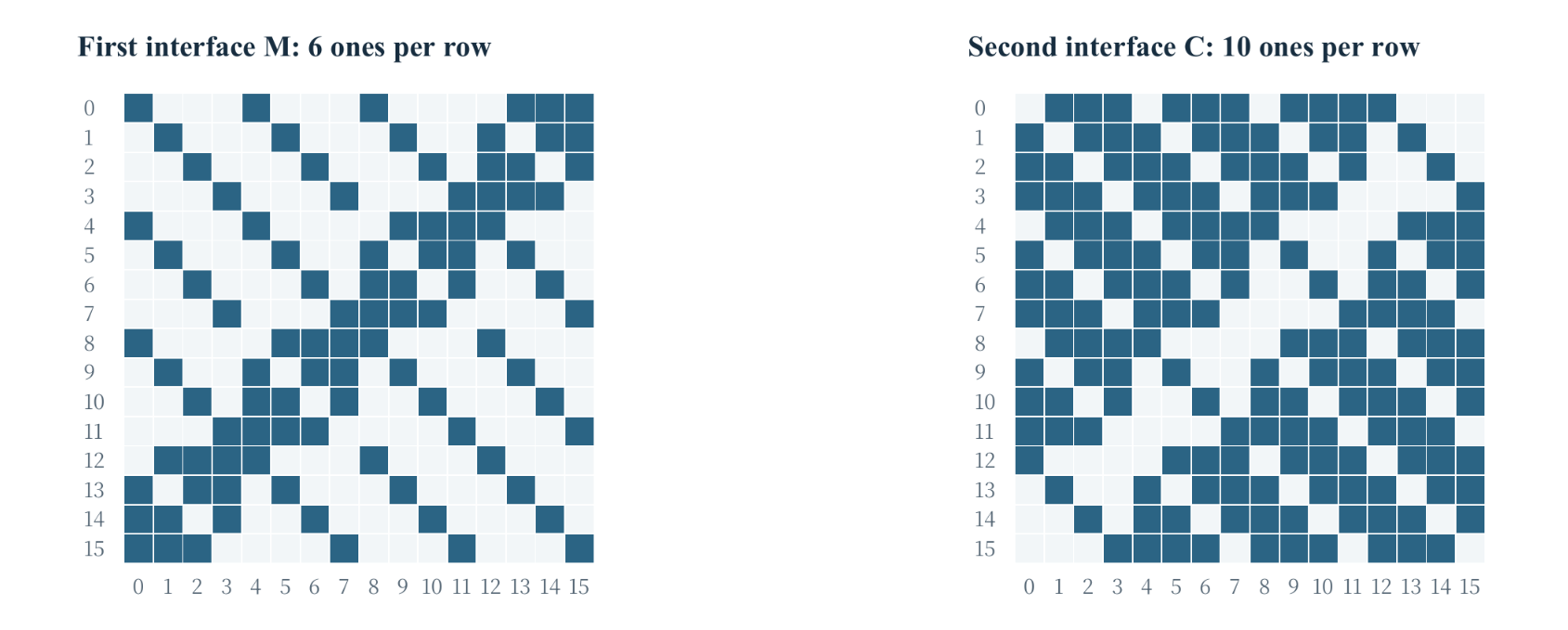}
	\includegraphics[width=1\linewidth]{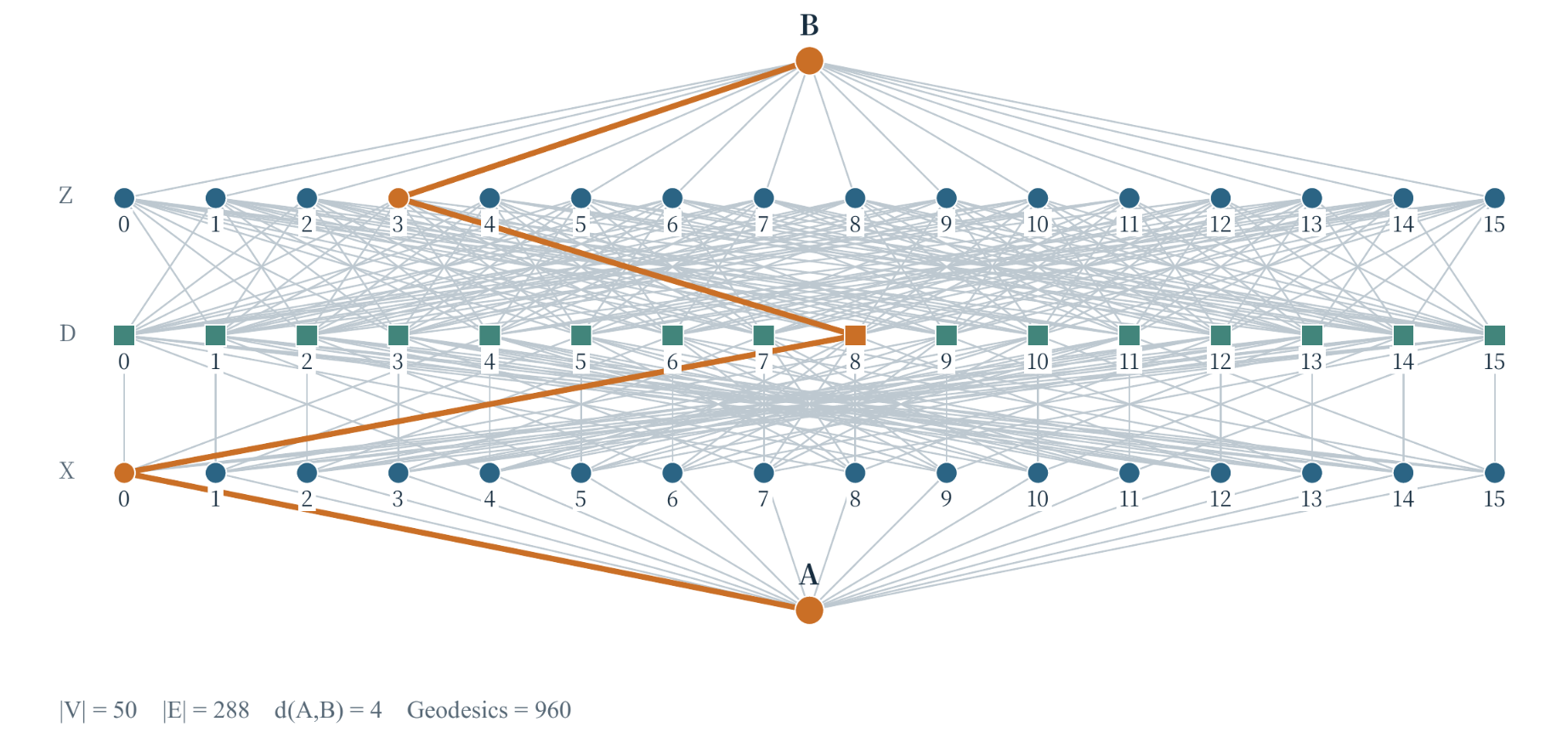}
	\caption{An example of A4}
	\label{fig:512}
\end{figure}

\section{The almost-simple case}\label{sec:almost-simple}

Throughout this section, the unique minimal normal subgroup $T=\soc(K)$ is assumed to be nonabelian simple, so
\[
T\leq K\leq\Aut(T).
\]
where the group $T$ is transitive on every internal layer, and $K$ acts faithfully on each layer.
Unlike the affine type, $T$ need not be regular, and the layer sizes need not be equal. If $x_i\in V_i$ and $H_i=K_{x_i}$, then
\[
V_i\cong K/H_i,\qquad K=TH_i,\qquad |V_i|=[K:H_i]=[T:T\cap H_i].
\]
These equalities proves that each layer is a transitive coset action of the same almost-simple group and widths determined by their point stabilizers.

%In this section, the unique minimal normal subgroup $T=\soc(K)$ is assumed to be nonabelian simple, so
%\[
%T\leq K\leq\Aut(T).
%\]
%Theorem~\ref{theorem1.1} gives transitivity of $T$ on every internal layer and faithfulness of the action of $K$ on each layer. Unlike the affine type, $T$ need not be regular, and the layer sizes need not be equal. If $x_i\in V_i$ and $H_i=K_{x_i}$, then
%\[
%V_i\cong K/H_i,\qquad K=TH_i,\qquad |V_i|=[K:H_i]=[T:T\cap H_i].
%\]
%These equalities say only that each layer is a transitive coset action of the same almost-simple group. They neither force the subgroups $H_i$ to be conjugate nor force the layers to have equal width.

\subsection{Unbounded reduced rank}\label{subsec:as-unbounded}

\begin{example}\label{thm:boolean-unbounded}
For every $n\geq5$, let $Q=B_n$ be the Boolean lattice of all subsets of $[n]$, ordered by inclusion, and let $K=S_n$ act naturally. Take $A=\varnothing$ and $B=[n]$. Then:
\begin{enumerate}
\item[(i)] $K$ is transitive on maximal chains and acts faithfully and 2-transitively on the first layer;
\item[(ii)] $T=A_n$ is transitive on every internal layer, so $(Q,K)$ is $K$-normal-basic;
\item[(iii)] Every adjacent internal interface is neither complete nor a perfect matching, so neither complete-cut decomposition nor matching compression reduces its rank;
\item[(iv)] The reduced rank is $n$, so there is no reduced rank bound in the almost-simple case that is independent of $n$.
\end{enumerate}
\end{example}

\begin{proof}
A maximal chain has the unique form
\[
\varnothing\subset\{a_1\}\subset\{a_1,a_2\}\subset\cdots\subset[n],
\]
where $(a_1,\ldots,a_n)$ is a permutation of $[n]$. 
The group $S_n$ is transitive on all permutations and therefore on maximal chains. 
The first layer consists of the subsets of one element and carries the natural faithful 2-transitive action of $S_n$.

We next prove that $A_n$ is transitive on every layer of $r$-subsets, where $1\leq r\leq n-1$. 
Given two $r$-subsets $R,S$, choose $\sigma\in S_n$ with $R^\sigma=S$. If $\sigma$ is even, it is the required element. 
If $\sigma$ is odd, choose an odd permutation $\tau$ in the stabilizer of $S$. 
When $|S|\geq2$, we can take a transposition in $S$. 
When $|S|=1$, the complement has size $n-1\geq4$, and we can take a transposition in that complement. 
In both cases, $\tau$ stabilizes $S$ and is odd. 
Hence $\sigma\tau\in A_n$ still maps $R$ to $S$. Thus $T=A_n$ is transitive on all layers. 
Since $n\geq5$, $A_n$ is nonabelian simple and is the socle of $S_n$. Theorem~\ref{theorem1.1} gives normal-basicness.

Finally, an $r$-subset $R$ has exactly $n-r$ upper covers, namely the sets $R\cup\{x\}$ with $x\notin R$, while an $(r+1)$-subset has exactly $r+1$ lower covers. 
For each internal interface $1\leq r\leq n-2$, both numbers are at least $2$, so the interface is not a matching. 
For a fixed $R$, there is an $(r+1)$-subset not containing $R$, the choices near the endpoints are also available directly because $n\geq5$. 
Thus there are no complete cuts. 
The reduced object is the original lattice and has rank $n$. Letting $n$ tend to infinity proves unboundedness.
\end{proof}

This example shows that first-layer 2-transitivity and normal-basicness can not compress the almost-simple case to finite local interfaces, that is to say the reduced rank is unbounded in the almost-simple case. 

\begin{figure}[H]
	\centering
	\includegraphics[width=0.8\linewidth]{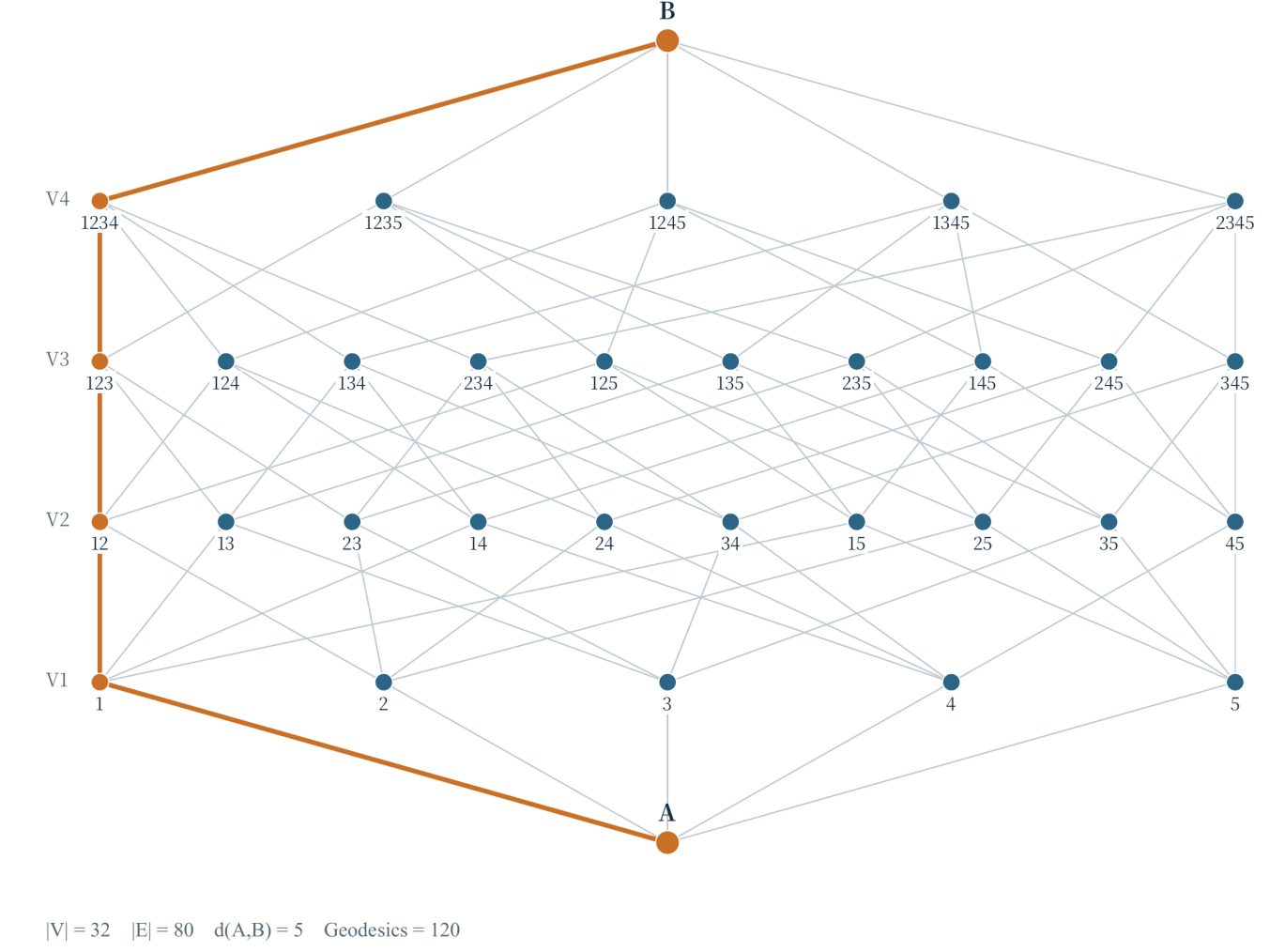}
	\caption{The example of Boolean lattice}
	\label{fig:61}
\end{figure}

The case \(n=5\) is shown in Figure \ref{fig:61}. 
Its four internal layers have sizes \(5,10,10,5\), and the lattice remains unchanged by both reduction operations. 
Although the reduced rank is unbounded in this family, the first-layer action still constrains the sizes of the other layers. 
The next proposition gives a lower bound for the size of any layer having a nonempty, noncomplete comparability relation with \(V_1\).
%However the first layer gives a lower bound for the width of every layer which has a noncomplete comparability relation with it.
%As stated in the following proposition:

\begin{proposition}\label{prop:first-layer-width}
Let $v=|V_1|$. 
If $j>1$ and the comparability relation $V_1$--$V_j$ is non-empty and non-complete, then $|V_j|\geq v$.
\end{proposition}

\begin{proof}
Let $B$ be the comparability matrix with rows indexed by $V_1$ and columns indexed by $V_j$. 
Since $K$ is 2-homogeneous on $V_1$, the number of common $1$ entries in any two distinct rows is a constant $\lambda$, and every row has a constant positive sum $r$.
Hence
\[
BB^{\mathsf T}=(r-\lambda)I+\lambda J.
\]
If $r=\lambda$, then every two rows have the same support. 
Since $K$ is transitive on $V_j$, all column sums are equal. 
The relation is nonempty, so every column is not zero, and the common row support must therefore contain all columns. 
This gives $B=J$, contrary to non-completeness. 
Thus $r-\lambda>0$. Let $c$ be the common column sum. 
$c>0$ because the relation is non-empty. 
On $\mathbf1^\perp$, the Gram matrix has eigenvalue $r-\lambda>0$, while
\[
BB^{\mathsf T}\mathbf1=B(c\mathbf1)=cr\mathbf1,
\]
so its eigenvalue on $\spanop\{\mathbf1\}$ is $cr>0$. 
Therefore $BB^{\mathsf T}$ is invertible and $\rank_{\mathbb{C}}(B)=v$. Since $B$ has $|V_j|$ columns,
\[
v=\rank(B)_{\mathbb{C}}\leq|V_j|.
\]
\end{proof}

\begin{corollary}\label{cor:equal-width-two-transitive}
Under the hypotheses of Proposition~\ref{prop:first-layer-width}, if $|V_j|=v$, then $B$ is invertible and gives a $K$-module isomorphism
\[
\mathbb C[V_j]\cong_K\mathbb C[V_1].
\]
So, $K$ is 2-transitive on $V_j$. If the relation degree lies in $[2,v-2]$, then the columns form a simple symmetric $2$-design.
\end{corollary}

\begin{proof}
The matrix $B$ is square and full rank, so it is invertible. 
The order relation is $K$-invariant, and hence $B$ intertwines the two permutation representations. 
Their permutation characters are therefore equal. 
The action on the first layer is 2-transitive, so the inner product of its permutation character is $2$. 
The action on $V_j$ is already known to be transitive, its permutation character also has inner product $2$. 
Equivalently, there are exactly two orbits on ordered pairs, the diagonal and the off-diagonal, so the action is 2-transitive. 
And degree is in $[2,v-2]$, invertibility of $B$ makes the columns linearly independent and in particular distinct.
Then from Lemma~\ref{lem:interface-design-identity}, a simple symmetric design is been given.
\end{proof}

\subsection{Equal-width blocks and interface compatibility}\label{subsec:anchored-equal-width}

Corollary \ref{cor:equal-width-two-transitive} gives a useful $2$-transitivity when layers are of the same sizes.
We now apply this condition to a proper block.

Let \(R\) be a proper block whose internal layers all have size \(v\), and suppose that \(K\) is \(2\)-transitive on at least one of them. 
The assumption holds whenever \(R\) contains \(V_1\). 
The following lemma shows that all internal layer actions then have the same permutation character and are \(2\)-transitive in this proper.

%If $R$ contains the first layer of the original poset, then in the present nonabelian-socle branch Theorem~\ref{thm:faithful-two-homogeneous} excludes the 2-homogeneous but non-2-transitive case, because that case is necessarily affine. Thus the original first layer is itself an anchor layer.

\begin{lemma}\label{lem:anchored-propagation}
In an block, $K$ is 2-transitive on every internal layer, all  permutation characters of the layers are equal, and there are exactly two $K$-orbits on the Cartesian product of any two internal layers. Every non-empty comparability relation is one of those orbits and is non-complete.
\end{lemma}

\begin{proof}
We suppose that $K$ is 2-transitive on $V_a$.
Consider an adjacent interface $V_i$--$V_{i+1}$ and first suppose that the action on one side (for example $V_i$), is known to be 2-transitive.
The interface is non-empty because every maximal chain crosses it, and due to it is a proper block there is no complete interface. 
Since the two layers have equal width, its matrix $B$ has equal row and column sums, denoted by $d$. 

If $d=1$, it is an matching, contrary to proper block. 
If $d=v-1$, then $B=J-P$ is invertible. If $2\leq d\leq v-2$, Lemma~\ref{lem:interface-design-identity} gives
\[
BB^{\mathsf T}=(d-\lambda)I+\lambda J
\]
and proves that $d-\lambda>0$, so $B$ is invertible.
The matrix $B$ is a $K$-module isomorphism and therefore carries the 2-transitive permutation module of $V_i$ isomorphically onto the permutation module of $V_{i+1}$. Hence the action on $V_{i+1}$ is also 2-transitive. 
If the action on the $V_{i+1}$ is known instead, the same argument applied to $B^{\mathsf T}$ propagates 2-transitivity to the $V_{i}$. 
Induction in both directions from $V_a$ shows that every layer is 2-transitive and all layer permutation characters are equal.

The number of cross-orbits on any two layers is the inner product of their permutation characters. 
Their common 2-transitive character has the form $1_K+\chi$ with $\chi$ irreducible, so this inner product is $2$. 
Fixed-type flag transitivity shows that comparable pairs of any fixed two ranks form one $K$-orbit, and this orbit is nonempty. If the $K$-orbit is the whole Cartesian product, there would be only one cross-orbit, contradicting the orbit count. Hence it is non-complete.
\end{proof}

\begin{lemma}\label{lem:two-step-arithmetic}
Let $X,Y,Z$ be three consecutive internal layers of a proper block, all of size $v$. Let $B,C$ be the two interface matrices, of degrees $a,b$, and $D$ is the $X$--$Z$ comparability matrix, of degree $d$. There is a positive integer $h$ such that
\[
BC=hD,
\]
and, with
\[
\nu(t):=\frac{t(v-t)}{v-1},
\]
one has
\[
h^2\nu(d)=\nu(a)\nu(b),\qquad hd=ab.
\]
Consequently,
\[
h=1+\frac{(a-1)(b-1)}{v-1},\qquad
d=\frac{ab(v-1)}{ab+v-a-b}.
\]
In particular, $h$ must be a positive integer.
\end{lemma}

\begin{proof}
	
From Proposition~\ref{prop:chain-counting-matrices}, we applied to flags of the three fixed ranks and get $BC=hD$, where $h$ is the number of intermediate points between a fixed comparable pair and is therefore a positive integer. Lemma~\ref{lem:anchored-propagation} gives 2-transitivity on all three layers, so
\[
BB^{\mathsf T}=\nu(a)I+\alpha J,\qquad
CC^{\mathsf T}=\nu(b)I+\beta J,\qquad
DD^{\mathsf T}=\nu(d)I+\gamma J.
\]
For example, $a-\lambda_a=\nu(a)$, since the design parameter identity
\[
\lambda_a(v-1)=a(a-1)
\]
gives
\[
a-\lambda_a=\frac{a(v-a)}{v-1}.
\]
The co-matching case $a=v-1$ satisfies the same formula. 
The value $a=1$ has been excluded by reduction, although the formula remains valid.

Now, considering that 
\[
(BC)(BC)^{\mathsf T}=B(CC^{\mathsf T})B^{\mathsf T}
=\nu(b)BB^{\mathsf T}+\beta BJB^{\mathsf T}.
\]
Since
\[
B\mathbf1=a\mathbf1,\qquad B^{\mathsf T}\mathbf1=a\mathbf1,
\]
the second term is a scalar multiple of $J$. Thus the coefficient of $I$ on the left is $\nu(a)\nu(b)$, while the coefficient of $I$ in $(hD)(hD)^{\mathsf T}$ is $h^2\nu(d)$. Since $I$ and $J$ are linearly independent, then we have gives
$h^2\nu(d)=\nu(a)\nu(b)$.
Comparison of the row sums in $BC=hD$ gives $ab=hd$.

Substitute $d=ab/h$ into the first identity after simplification, we obtain 
\[
\frac{(v-a)(v-b)}{(v-1)^2}=\frac{hv-ab}{v-1}.
\]
Multiplying by $(v-1)^2$ and yields
\[
hv(v-1)=ab(v-1)+(v-a)(v-b)=v(ab+v-a-b).
\]
After cancellation of $v$,
\[
h=\frac{ab+v-a-b}{v-1}=1+\frac{(a-1)(b-1)}{v-1}.
\]
Finally, $d=ab/h$ gives the formula for $d$.
\end{proof}

%\subsection{Compatibility of interfaces}\label{subsec:kantor-compatibility}

\begin{proposition}\label{prop:almost-simple-interface-candidates}
Every adjacent interface in a proper block is exactly either a co-matching or one of designs (or its complement) in almost simple case in Lemma~\ref{thm:two-transitive-symmetric-designs}.
\end{proposition}

\begin{proof}
Let the interface degree be $a$.
Matching reduction excludes $a=1$, and properness excludes $a=v$.
If $a=v-1$, the interface is a co-matching. 
If $2\leq a\leq v-2$, Lemmas~\ref{lem:interface-design-identity} and~\ref{lem:anchored-propagation} give a simple symmetric $2$-design and $K$ acts faithfully and 2-transitively on points.
As in the proof of Proposition~\ref{prop:affine-interface-candidates}, the cases $a=2$ and $v-a=2$ is the co-matching at $v=3$. 
Hence a ture intermediate case satisfies $3\leq\min\{a,v-a\}$.
So we can applies Lemma~\ref{thm:two-transitive-symmetric-designs}. 
Since the socle of $K$ is nonabelian simple, its affine symplectic case is excluded, leaving exactly the three almost simple case and their complements.
\end{proof}

\begin{lemma}\label{lem:common-socle-kantor-row}
All design interfaces of the same proper block belongs to the same group-labelled Kantor case. 
In particular, their block sizes can differ only by interchanging $\kappa$ and $v-\kappa$.
\end{lemma}

\begin{proof}
We first check that the condition of applying group-labelled classification.
Consider a design interface between consecutive layers \(X\) and \(Y\). Take \(X\) as the point set and, for each \(y\in Y\), define
\[
B_y=\{x\in X: x<y\}.
\]
The columns of the incidence matrix are distinct, so different elements of \(Y\) give different blocks \(B_y\). 
We may therefore identify \(Y\) with the block set \(\mathcal B=\{B_y\in Y\}\) and $K\leq\Aut(\mathcal D)$ is faithful and \(2\)-transitive.Thus \(K\) is \(2\)-transitive on both points and blocks. 

There are exactly two \(K\)-orbits on \(X\times\mathcal B\). The incident point–block pairs form one orbit by Lemma \ref{lem:fixed-type-flags}, so the nonincident pairs form the other.
Since any two distinct blocks of this symmetric design intersect, these transitivity properties show that \(\mathcal D\) is \(K\)-pairwise transitive. 
So, the classification of Devillers–Praeger [5, Proposition 3.5 and Table 1] therefore applies.

In that group-labelled table, the affine symplectic case has a regular elementary abelian socle and is incompatible with the hypotheses of this section.
For the remaining entries, the labels given by Devillers--Praeger~\cite[Proposition 3.5 and Table 1]{devillers2015} can be arranged as follows:
\[
\begin{array}{c|c|c}
\text{design case}&\text{point-action degree}&\soc(K)\\ \hline
PG(d-1,q)&(q^d-1)/(q-1)&\PSL_d(q)\ (d\geq3)\\
\text{additional action on }PG(3,2)&15&A_7\\
2\text{-}(11,5,2)&11&\PSL_2(11)\\
\text{Higman--Sims}&176&HS
\end{array}
\]
%For $(d,q)=(4,2)$, the usual group label in the first row has socle $\PSL_4(2)\cong A_8$. The $A_7$ label is a different group label on the same design parameters and must not be conflated with it.

We now eliminate all possible coincidences. 
The two interface actions have the same faithful group $K$, and hence the same socle $T=\soc(K)$.
Lemma~\ref{lem:anchored-propagation} also gives the same action degree $v$. 
The 11-point and 176-point case have different degrees. 
Neither can coincide with a projective case. 
If
\[
v=1+q+\cdots+q^{d-1}\qquad(d\geq3),
\]
then $q\mid v-1$. 
For $v=11$, the only prime powers dividing $10$ are $2$ and $5$.
When $q=2$, the degrees for $d=3,4$ are $7,15$ and then increase;
when $q=5$, the smallest degree is already $31$. 

For $v=176$, the only prime powers dividing $175$ are $5,7,25$. For $q=5$, the degrees begin $31,156,781$; for $q=7$, they begin $57,400$; and for $q=25$, the smallest degree is $651$. 
None equals $176$. 
Thus there is no equal-degree crossing among the three design families.

It remains to consider the case in which both interfaces belong to the projective case.
After stating the full list of finite simple groups, Wilson~\cite[Chapter 1, formula (1.2)]{wilson2009} gives all duplicate names in that list:
\[
\PSL_2(4)\cong\PSL_2(5)\cong A_5,\qquad
\PSL_2(7)\cong\PSL_3(2),\qquad
\PSL_2(9)\cong A_6,
\]
\[
\PSL_4(2)\cong A_8,\qquad
\operatorname{PSU}_4(2)\cong\operatorname{PSp}_4(3),
\]
Suppose first that the two interfaces arise from the usual projective actions on \(\operatorname{PG}(d-1,q)\) and \(\operatorname{PG}(e-1,r)\), where \(d,e\geq3\). 
Since both actions have socle \(T\cong\operatorname{PSL}_d(q)\cong\operatorname{PSL}_e(r)\).
The list above show that \(d=e\) and \(q=r\). 
Indeed, the exception \(\operatorname{PSL}_3(2)\cong\operatorname{PSL}_2(7)\) has dimension \(2\) on one side, whereas both dimensions here are at least \(3\). 
The isomorphism \(\operatorname{PSL}_4(2)\cong A_8\) gives another name for the same group, not a different pair of projective parameters.

Now suppose that one interface arises from the additional \(A_7\)-action on \(\operatorname{PG}(3,2)\). Then \(T\cong A_7\).
Since \(A_7\) is not isomorphic to any \(\operatorname{PSL}_e(r)\) with \(e\geq3\), the other interface must also arise from the \(A_7\)-action on \(\operatorname{PG}(3,2)\).

Thus, after taking complements where necessary, the two interfaces are point–hyperplane designs of the same projective space and have the same parameters \((v,\kappa,\lambda)\). Restoring the original interfaces replaces the block size \(\kappa\) by \(v-\kappa\) whenever a complement was taken. Hence their block sizes belong to \({\kappa,v-\kappa}\).

%
%If two projective entries have a common socle
%$T\cong\PSL_d(q)\cong\PSL_e(r)$ with $d,e\geq3$, this list forces $(d,q)=(e,r)$. The only linear--linear duplication involving different parameters and one dimension at least $3$ is $\PSL_3(2)\cong\PSL_2(7)$, whose other dimension is $2$ and is outside the present range. The isomorphism $\PSL_4(2)\cong A_8$ does not supply a second linear parameter pair of dimension at least $3$. If one interface uses the additional $A_7$ label on $PG(3,2)$, then the common socle is $A_7$; the complete list shows that it is not isomorphic to any $\PSL_e(r)$, so the other interface must use the same $A_7$ entry. Hence, under a fixed $T$, the projective part leaves only the same $(d,q)$ entry, or the same $A_7$ entry. Thus the two interfaces have the same parameters $(v,\kappa,\lambda)$. If a complement was taken before applying Kantor's classification, the only remaining freedom is to replace $\kappa$ by $v-\kappa$. This proves the assertion.
\end{proof}

\begin{lemma}\label{lem:same-side-exclusion}
Two consecutive interfaces cannot both have degree $\kappa$, and cannot both have degree $v-\kappa$.
\end{lemma}

\begin{proof}
By Lemma~\ref{lem:two-step-arithmetic}, consecutive degrees $a,b$ must make
\[
h=1+\frac{(a-1)(b-1)}{v-1}
\]
an integer. We check the cases one by one.

For the projective case,
\[
v-1=q\kappa,\qquad \kappa-1=q\lambda,\qquad
\gcd(\kappa,q)=\gcd(\kappa,\lambda)=1.
\]
The last equality follows from $\kappa-q\lambda=1$. If $a=b=\kappa$, then
\[
h=1+\frac{q^2\lambda^2}{q\kappa}=1+\frac{q\lambda^2}{\kappa},
\]
which is not an integer because $\kappa\nmid q\lambda^2$. If $a=b=v-\kappa=q^{d-1}$, then $a-1=(q-1)\kappa$, and hence
\[
h=1+\frac{(q-1)^2\kappa^2}{q\kappa}
=1+\frac{(q-1)^2\kappa}{q}.
\]
Since $q\nmid(q-1)^2\kappa$, this is again not an integer.

For the $2$-$(11,5,2)$ case, the two same side values give
\[
h(5,5)=1+\frac{16}{10}=\frac{13}{5},\qquad
h(6,6)=1+\frac{25}{10}=\frac72.
\]
For the Higman--Sims case,
\[
h(50,50)=1+\frac{49^2}{175}=\frac{368}{25},\qquad
h(126,126)=1+\frac{125^2}{175}=\frac{632}{7}.
\]
None is an integer. 
These three cases have exhausted the nondegenerate possibilities in Proposition~\ref{prop:almost-simple-interface-candidates}.
\end{proof}

The following lemma is needed for the Higman-Sims construction in AS4.

\begin{lemma}\label{lem:mixed-interface-uniqueness}
Suppose that two consecutive design interfaces have degrees $\kappa,v-\kappa$, in either order. Let $\delta=\kappa-\lambda$.
%Then the distant relation is a co-matching and the chain multiplicity is $h=\delta$. After the $K$-equivariant identification of the two outer layers supplied by this co-matching, 
And if the first interface matrix is $B$, then the second is
\[
C=J-B^{\mathsf T},\qquad BC=\delta(J-I).
\]
No further interface can occur before or after this pair.
\end{lemma}

\begin{proof}
If the interface order is $v-\kappa,\kappa$, reverse the order of the three layers and transpose all relation matrices. This reduces the proof to the order $\kappa,v-\kappa$, transposing the resulting identities gives the original order. 
We therefore assume that the first interface has degree $\kappa$.

The symmetric design identity gives that
\[
\delta(v-1)=(\kappa-\lambda)(v-1)
=\kappa(v-1)-\kappa(\kappa-1)=\kappa(v-\kappa).
\]
In Lemma~\ref{lem:two-step-arithmetic}, take $a=\kappa$ and $b=v-\kappa$. Rearrangement gives
\[
h=1+\frac{(\kappa-1)(v-\kappa-1)}{v-1}=\delta
\]
and
\[
d=\frac{\kappa(v-\kappa)}{\delta}=v-1.
\]
Thus the distant matrix is $J-P$. Its zero in each row and column supplies a bijection between the outer layers. 
Label through that bijection so that $P=I$. 
On one hand, chain counting then gives $BC=\delta(J-I)$.

On the other hand, the design identity is $BB^{\mathsf T}=\delta I+\lambda J$, and therefore
\[
B(J-B^{\mathsf T})=\kappa J-(\delta I+\lambda J)=\delta(J-I).
\]
The matrix $B$ is invertible, so we yields $C=J-B^{\mathsf T}$ and proves uniqueness.

If a further interface matrix $E$ of degree $e\geq2$ occurs on the right, then
\[
BCE=\delta(J-I)E=\delta(eJ-E)
\]
is strictly positive entrywise.
Hence the two outer layers are completely comparable, contrary to the distant non-completeness in Lemma~\ref{lem:anchored-propagation}. 
If the third interface occurs on the left, reverse the layers and transpose all interface matrices, the resulting matrix $\delta(eJ-E^{\mathsf T})$ is still strictly positive, giving the same contradiction.
\end{proof}

\subsection{Classification and constructions}
\begin{lemma}\label{prop:hs-flag-stabilizer}
	
%	The Higman--Sims case of AS4 requires only the following precise local statement.
	
		Let $G=HS$ act on the Higman--Sims symmetric $2$-$(176,50,14)$ design. Fix an incident flag $(x,Y)$. Then $G_{x,Y}$ is transitive on the 126 points outside $Y$.
	
	\begin{proof}
		Brouwer's construction of this design and calculation of its flag stabilizer~\cite{brouwer1982} give
		\[
		A:=G_{x,Y}\cong S_7.
		\]
		The ATLAS data for the maximal subgroups of $HS$ and its 176-point representation~\cite{atlas} identify the block stabilizer as
		\[
		H:=G_Y\cong U_3(5):2,\qquad |H|=252000.
		\]
		The correspondence between the block and point sides is supplied by the polarity in Brouwer's construction. 
		We have already proved that incident and nonincident point--block pairs form two $G$-orbits. 
		Hence the two orbits of $H$ on the point set are exactly $Y$ and its complement, of lengths $50$ and $126$. 
		Fix $z\notin Y$ and put $B:=H_z$. We have
		\[
		|B|=|H|/126=2000.
		\]
		The subgroup-containment information for the same ATLAS representation identifies this stabilizer as
		\[
		B\cong5_+^{1+2}:(8:2).
		\]
		%We now specify exactly which entry of the table of Liebeck--Praeger--Saxl is used. 
		Let
		$
		H_0:=U_3(5)\trianglelefteq H.
		$
		Also from ATLAS, we obtain
		\[
		A_0:=A\cap H_0\cong A_7,\qquad
		B_0:=B\cap H_0\cong P_1=5_+^{1+2}:8.
		\]
		These subgroups have index $2$ in $A$ and $B$, respectively. The same data place $A_0$ and $B_0$ in the $A_7$ class and the parabolic class $P_1=5_+^{1+2}:8$ of $H_0$, respectively. The relative position of these actual subgroups must now be checked; their isomorphism types alone do not imply a product factorization.
		
		For representatives $A\cong A_7$ and $P\cong P_1$ of these two conjugacy classes, from \cite{liebeck1990}, we can obtain
		\[
		H_0=AP,
		\qquad |A\cap P|=20.
		\]
		Simultaneous conjugation of the present pair by an element of $H_0$ changes neither the order of the intersection nor whether its product is all of $H_0$, so we may take $A_0=A$. Since $B_0$ lies in the conjugacy class of $P$, there exists $h\in H_0$ such that $B_0=hPh^{-1}$.
		
		The basic factorization $H_0=A_0P$ is equivalent to transitivity of $A_0$ on the left coset space $H_0/P$. The stabilizer in $A_0$ of an arbitrary point $hP$ is
		\[
		(A_0)_{hP}=A_0\cap hPh^{-1}=A_0\cap B_0,
		\]
		and the orbit has length
		\[
		[H_0:P]=126.
		\]
		The orbit--stabilizer theorem therefore gives
		\[
		|A_0\cap B_0|=|A_0|/126=2520/126=20.
		\]
		Consequently,
		\[
		|A_0B_0|=\frac{|A_0||B_0|}{|A_0\cap B_0|}
		=\frac{2520\cdot1000}{20}=126000=|H_0|,
		\]
		and hence $H_0=A_0B_0$. 
		
		We now lift this basic table entry step by step to $H$. Since $[A:A_0]=2$, choose $a\in A\setminus H_0$. The normal subgroup $H_0$ has index $2$ in $H$, so
		\[
		H=H_0\mathbin{\dot\cup}aH_0.
		\]
		Using $H_0=A_0B_0$, together with $A_0\leq A$ and $B_0\leq B$, gives
		\[
		H_0=A_0B_0\subseteq AB,
		\qquad aH_0=aA_0B_0\subseteq AB.
		\]
		Thus $H=AB$. The product formula for finite groups now gives
		\[
		|A\cap B|=\frac{|A||B|}{|H|}
		=\frac{5040\cdot2000}{252000}=40.
		\]
		Therefore the orbit of $z$ under $A=G_{x,Y}$ in the coset space $H/B$ has length
		\[
		[A:A\cap B]=5040/40=126.
		\]
		The coset space $H/B$ is exactly the outer point orbit $\Omega\setminus Y$ of $H$, which also has size $126$. Hence this orbit is the entire set of points outside $Y$, proving the proposition.

	\end{proof}
	
\end{lemma}

%\begin{theorem}\label{thm:almost-simple-classification}
%Every proper block has exactly one of the following four shapes:
%
%\medskip
%\noindent\textbf{AS1 (one layer)} one internal layer;
%
%\noindent\textbf{AS2 (co-matching)} two internal layers whose interface is $J-P$;
%
%\noindent\textbf{AS3 (one design)} two internal layers whose interface is a projective design, the $2$-$(11,5,2)$ design, or the Higman--Sims design, or the complement of one of these;
%
%\noindent\textbf{AS4 (mixed complementary pair)} three internal layers whose consecutive interfaces have matrices $B$ and $J-B^{\mathsf T}$ for the same one of the preceding design cases.
%
%Thus the reduced rank is at most $4$ and every listed shape occurs.
%\end{theorem}
Based on the above results, we will now present the proof of Theorem \ref{thm:main-almost-simple} and demonstrate the existence of each type.

\begin{proof}
Necessity discuss separately according to the number of interfaces. 
With no interface one obtains AS1. 
With exactly one interface, Proposition~\ref{prop:almost-simple-interface-candidates} gives either a co-matching or one of the three Kantor designs, producing AS2 or AS3, respectively. 

If there are at least two interfaces, a co-matching cannot occur, because Lemma~\ref{lem:comatching-isolation} proves that would make the relation complete, contrary to Lemma~\ref{lem:anchored-propagation}. 
Thus the first two interfaces are both nondegenerate designs.
Lemma~\ref{lem:common-socle-kantor-row} places them in the same case; Lemma~\ref{lem:same-side-exclusion} excludes the same degree; and Lemma~\ref{lem:mixed-interface-uniqueness} forces a complementary pair and excludes any third interface. 
This gives AS4. 
So, the necessity classification is therefore complete.

Next, we verify existence type by type.

For AS1, let any faithful almost-simple 2-transitive group $K$ act on $X$ and suspend the endpoints. Maximum chains correspond bijectively to $X$. The socle $T=\soc(K)$ is transitive on $X$, so it is normal-basicness.

For AS2, take two sets $X,Z$ and a $K$-equivariant bijection $\varphi$, and define $x<z$ if and only if $z\neq\varphi(x)$.
Maximal chains correspond to ordered pairs of distinct points of $X$, and 2-transitivity of $K$ gives chain transitivity.
The group $T$ is transitive on both layers and the first-layer action is faithful, so the object is normal-basic. 

For AS3, take respectively
\[
K=\PGL_d(q),\qquad K=\PSL_2(11),\qquad K=HS
\]
in their standard actions on the corresponding projective, 11-point, and 176-point designs. Use the same group for the complementary design.  
These groups are 2-transitive on both points and blocks and that incident flags and non-incident point-block pairs each form one orbit. 
Hence, the maximal chains, which are the design flags, form one orbit. 
The respective socles $\PSL_d(q),\PSL_2(11),HS$ are transitive on both the point and block layers. 
The socle criterion gives normal-basicness.

For AS4, maximal chains are the triples
\[
(x,Y,z),\qquad x\in Y,\quad z\notin Y.
\]
First verify the outer relation uniformly. In any symmetric $2$-$(v,\kappa,\lambda)$ design, the same point $x=z$ cannot satisfy both $x\in Y$ and $z\notin Y$. If $x\neq z$, then among the $\kappa$ blocks containing $x$, exactly $\lambda$ also contain $z$. Hence precisely $\kappa-\lambda>0$ blocks contain $x$ and omit $z$. Thus, in all three design families, the transitive closure induces exactly a co-matching between the two outer point layers. It remains to verify transitivity on such triples case by case.

For the projective case, let $X,Z$ be the point set of $PG(d-1,q)$ and let $Y$ be the hyperplane set. 
Given a triple, choose representative vectors $u,w$ with $x=\langle u\rangle$ and $z=\langle w\rangle$. 
If $Y=\ker f$, then $f(u)=0$ and $f(w)\neq0$. Extend $u$ to a basis of $Y$ and adjoin $w$ to obtain a basis of the whole vector space. 
Any two triples are mapped to one another by an invertible linear transformation sending one such basis to the other. 
Hence $\PGL_d(q)$ is transitive on the three-term maximal chains. A hyperplane containing $x$ but not $z$ exists precisely when $x\neq z$, confirming again that the outer relation is a co-matching. 
The socle $\PSL_d(q)$ is transitive on both points and hyperplanes and hence on every layer.

For the 11-point case, fix a block $Y$. Its stabilizer $H=K_Y\cong A_5$ has orbits of lengths $5$ and $6$ on $Y$ and its complement. For $x\in Y$ and $z\notin Y$,
\[
H_x\cong A_4,\qquad H_z\cong D_{10}.
\]
Identify $H$ as  $A_5$ on five elements. The subgroup $D_{10}$ is the normalizer of a $5$-cycle. 
Its five involutions are double transpositions, each fixes exactly one of the five elements, and these fixed elements are all distinct. 
Hence $H_x\cap H_z$ consists of the identity and the unique involution fixing $x$, and has order $2$. 
Therefore
\[
|H_xH_z|=\frac{|H_x||H_z|}{|H_x\cap H_z|}=\frac{12\cdot10}{2}=60=|H|,
\]
so $H=H_xH_z$. Equivalently, $H_x$ is transitive on $H/H_z$, that is, on the six points outside $Y$. Since $H$ is transitive on $Y$, it is transitive on $Y\times(\Omega\setminus Y)$. 
Together with block transitivity of $K$, this proves transitivity on all triples. The group $K=\PSL_2(11)$ is simple and transitive on all three layers.

For a concrete description of this \(11\)-point construction, take the point set to be \(\mathbb F_{11}\) and the blocks to be
\[
D_t=t+\{1,3,4,5,9\},\qquad t\in\mathbb F_{11}.
\]
Using incidence below the block layer and nonincidence above it gives the poset in Figure \ref{fig:610}. 
Its two interfaces have degrees \(5\) and \(6\), and each pair of distinct outer labels has exactly \(5-2=3\) intermediate lines.
\begin{figure}[H]
	\centering
	\includegraphics[width=1\linewidth]{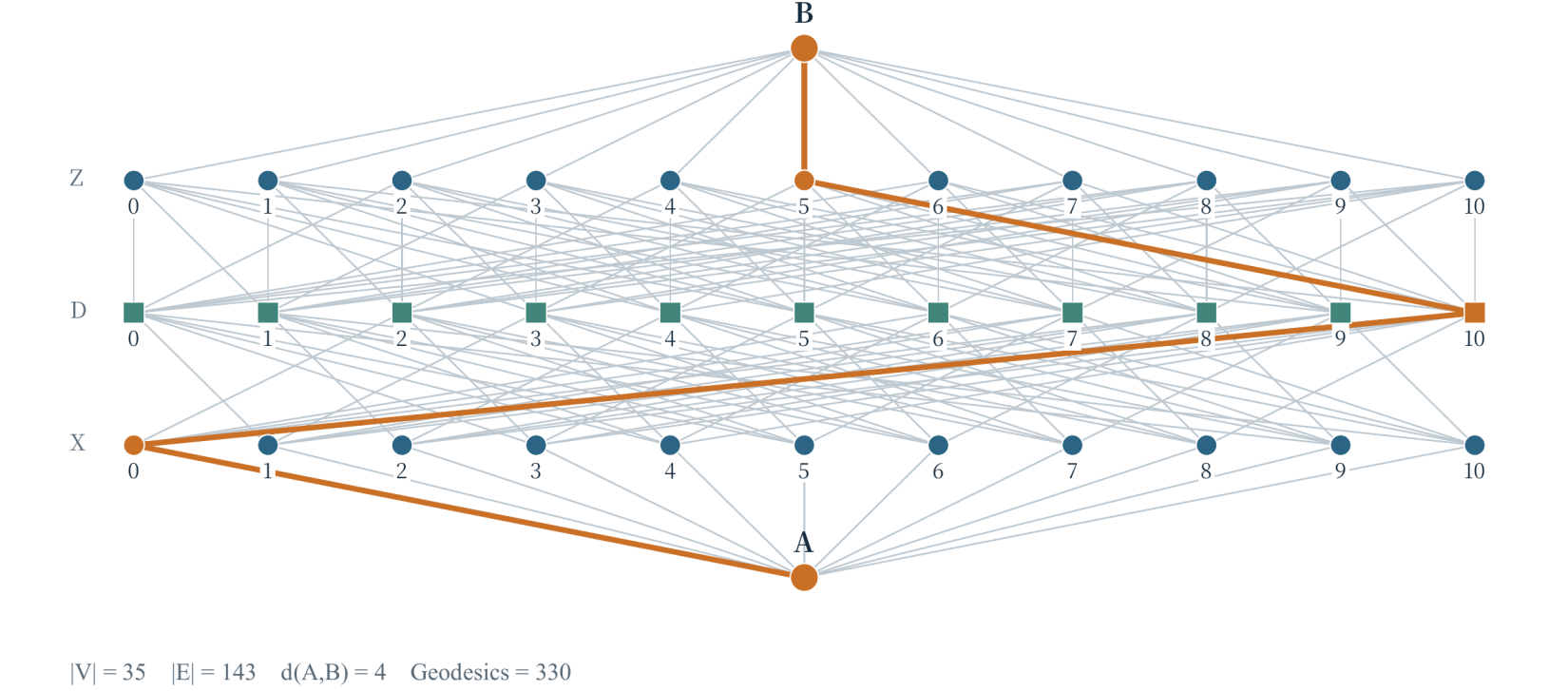}
	\caption{The example of the 11-point case}
	\label{fig:610}
\end{figure}

For the Higman--Sims case, whose original 176-point design construction is due to Higman~\cite{higman1967}, Lemma~\ref{prop:hs-flag-stabilizer} proves that, after an incident flag $(x,Y)$ is fixed, its stabilizer $K_{x,Y}$ is transitive on the 126 points outside $Y$. 
Thus any triple $(x,Y,z)$ can first be sent to a fixed incident flag using point--block flag transitivity and then have its third entry moved by the flag stabilizer. 
All three-term maximal chains therefore form one $K$-orbit. 
%The proof of that proposition starts from the factorization $U_3(5)=A_7P_1$ actually listed in the tables of Liebeck--Praeger--Saxl and then lifts it rigorously to $U_3(5):2$. 
The group $K=HS$ is simple and transitive on all three 176-element layers.

The constructions above place the design interface of degree $\kappa$ on the left and its complementary dual of degree $v-\kappa$ on the right. The reverse order requires no additional group-theoretic hypothesis. 
Take the order dual of the three-layer poset and interchange the two suspended endpoints. If the original interface matrices are $B,J-B^{\mathsf T}$, then the reversed matrices are $J-B,B^{\mathsf T}$, with degrees $v-\kappa,\kappa$, exactly the complementary design and the dual of the original design.
Order duality gives a $K$-equivariant bijection between the two maximal chain sets and preserves Hasse-graph distances, maximal-chain transitivity, transitivity of the socle on all internal layers, and the normal-basic quotient condition. 
The new first layer is the other point set and carries the same faithful 2-transitive point action. Thus both interface orders are realized.

In all three AS4 families, $T$ is transitive on every internal layer and the first-layer action is faithful and 2-transitive. 
Theorem~\ref{theorem1.1} gives normal-basicness. 
Every cover relation belongs to the already established transitive orbit of maximal chains, and the endpoint distance in the Hasse graph equals the rank. 
Proposition~\ref{thm:core-hasse} therefore shows that these are indeed endpoint-geodesic cores. %Proposition~\ref{prop:matching-compression} supplies all equivariant matching subdivisions.

\begin{figure}[H]
	\centering
	\includegraphics[width=1\linewidth]{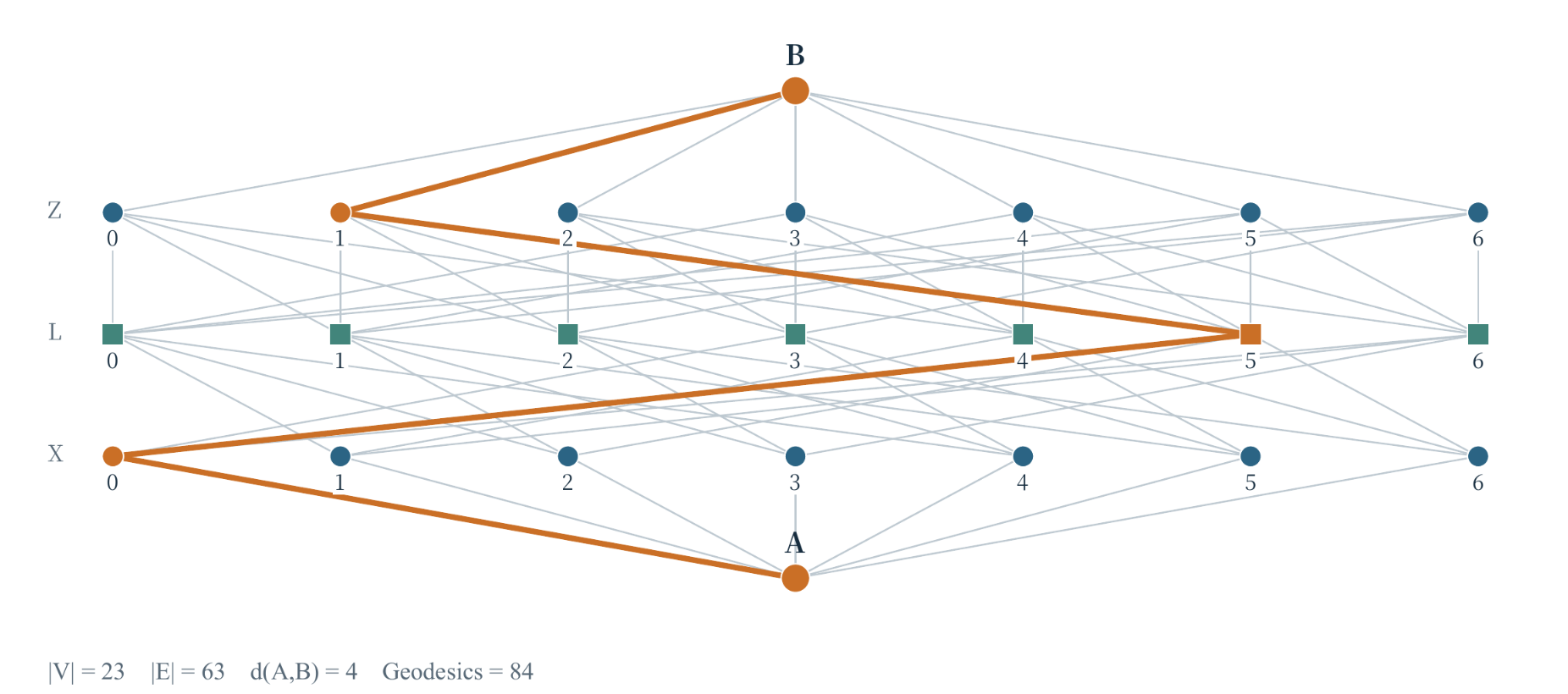}
	\caption{The example of complementary interfaces}
	\label{fig:610as4}
\end{figure}
Taking \(d=3\) and \(q=2\) in this projective construction gives the block in Figure \ref{fig:610as4}. The Fano plane from Figure \ref{fig:5} supplies the first interface; the second joins each line to the points outside it in the upper point layer. The two interfaces therefore have degrees \(3\) and \(4\). If \(N\) is the point–line incidence matrix, then
\[
N(J-N^{\mathsf T})=2(J-I),
\]
so each pair of distinct outer labels has exactly two lines.
\end{proof}
\begin{remark}\label{rem:equal-width-boundary}
Let $m\geq2$ and consider the layer of $m$-subsets and the layer of $(m+1)$-subsets in $B_{2m+1}$. The two layers have equal size and their interface is nontrivial, but $S_{2m+1}$ is not 2-transitive on the $m$-subsets: ordered pairs of distinct $m$-subsets can be distinguished, for example, by intersection sizes $m-1$ and $m-2$. This adjacent equal-width pair lies inside a proper block whose layer widths are not constant. Therefore an equal-width interface cannot replace the all-segment equal-width assumption.
%in Definition~\ref{def:anchored-equal-width}. When $m=1$, the middle interface of $B_3$ is merely the co-matching AS2 and is not a boundary example.
\end{remark}

\begin{appendices}

\end{appendices}

\renewcommand{\refname}{References}

\end{document}